\documentclass[10pt,twoside,english]{amsart}
\normalsize
\advance\oddsidemargin by -1.0cm
\advance\evensidemargin by -1.0cm
\usepackage{amssymb}
\usepackage{babel}
\usepackage{amsmath}
\usepackage{amscd}
\usepackage{epsfig}
\usepackage{psfrag}
\let\mathg\mathfrak
\theoremstyle{plain}
\newtheorem{cor}{Corollary}[section]
\newtheorem{lem}{Lemma}[section]
\newtheorem{thm}{Theorem}[section]
\newtheorem{prop}{Proposition}[section]
\theoremstyle{definition}
\newtheorem{exa}{Example}[section]
\newtheorem{NB}{Remark}[section]

\newtheorem{dfn}{Definition}[section]

\renewcommand{\labelenumi}{$(\theenumi)$}
\newcommand{\bdm}{\begin{displaymath}}
\newcommand{\edm}{\end{displaymath}}
\newcommand{\be}{\begin{equation}}
\newcommand{\ee}{\end{equation}}
\newcommand{\ba}[1]{\begin{array}{#1}}
\newcommand{\ea}{\end{array}}
\newcommand{\bea}[1][]{\begin{eqnarray#1}}
\newcommand{\eea}[1][]{\end{eqnarray#1}}
\newcommand{\bqr}{\begin{eqnarray}}
\newcommand{\eqr}{\end{eqnarray}}
\newcommand{\bqrs}{\begin{eqnarray*}}
\newcommand{\eqrs}{\end{eqnarray*}}
\newcommand{\btab}{\begin{tabular}}
\newcommand{\etab}{\end{tabular}}

\newcommand{\x}{\times}

\newcommand{\ox}{\otimes}

\newcommand{\ra}{\rightarrow}

\newcommand{\lan}{\left\langle}
\newcommand{\ran}{\right\rangle}

\newcommand{\Id}{\ensuremath{\mathrm{Id}}}
\newcommand{\tr}{\ensuremath{\mathrm{tr}}}

\newcommand{\univ}{\ensuremath{\mathrm{univ}}}
\newcommand{\tw}{\ensuremath{\mathrm{tw}}}

\newcommand{\R}{\ensuremath{\mathbb{R}}}

\newcommand{\M}{\ensuremath{\mathcal{M}}}

\newcommand{\D}{\slash{\!\!\!\!D}}

\newcommand{\T}{\ensuremath{\mathcal{T}}}
\newcommand{\TM}{\ensuremath{\mathcal{T}\!M}}
\newcommand{\Hol}{\ensuremath{\mathrm{Hol}}}

\newcommand{\vphi}{\ensuremath{\varphi}}

\newcommand{\hut}{\wedge}

\newcommand{\kr}{\ensuremath{\mathcal{R}}}
\newcommand{\Ric}{\ensuremath{\mathrm{Ric}}}
\newcommand{\Scal}{\ensuremath{\mathrm{Scal}}}
\newcommand{\Scalg}{\ensuremath{\mathrm{Scal}^g}}
\newcommand{\Scalgmin}{\ensuremath{\mathrm{Scal}^g_{\min}}}

\newcommand{\Ad}{\ensuremath{\mathrm{Ad}\,}}
\newcommand{\Adtilde}{\ensuremath{\widetilde{\mathrm{A}}\mathrm{d}\,}}

\newcommand{\diag}{\ensuremath{\mathrm{diag}}}

\newcommand{\GL}{\ensuremath{\mathrm{GL}}}

\newcommand{\so}{\ensuremath{\mathg{so}}}
\newcommand{\SO}{\ensuremath{\mathrm{SO}}}
\newcommand{\Spin}{\ensuremath{\mathrm{Spin}}}

\newcommand{\Orth}{\ensuremath{\mathrm{O}}}

\newcommand{\m}{\ensuremath{\mathfrak{m}}}

\begin{document}
\def\haken{\mathbin{\hbox to 6pt{%
                 \vrule height0.4pt width5pt depth0pt
                 \kern-.4pt
                 \vrule height6pt width0.4pt depth0pt\hss}}}
    \let \hook\intprod
\setcounter{equation}{0}
%
%
\thispagestyle{empty}
%
\title[Twistorial eigenvalue estimates]{Twistorial eigenvalue estimates
for generalized Dirac operators with torsion}
%
%
%
\author{Ilka Agricola}
\author{Julia Becker-Bender}
\author{Hwajeong Kim}
\address{\hspace{-5mm}
Ilka Agricola, Julia Becker-Bender \newline
Fachbereich Mathematik und Informatik \newline
Philipps-Universit\"at Marburg\newline
Hans-Meerwein-Strasse \newline
D-35032 Marburg, Germany\newline
{\normalfont\ttfamily agricola@mathematik.uni-marburg.de}\newline
{\normalfont\ttfamily beckbend@mathematik.uni-marburg.de}}

\address{\hspace{-5mm}
Hwajeong Kim \newline
Department of Mathematics\newline
Hannam University\newline
Daejeon 306-791, Republic of Korea\newline
{\normalfont\ttfamily hwajkim@hnu.kr}\newline
}
%
%
%
\pagestyle{headings}
\begin{abstract}
We study the Dirac spectrum on compact
Riemannian spin manifolds $M$ equipped with a metric connection $\nabla$ with skew 
torsion $T\in\Lambda^3M$ by means of twistor theory.
An optimal lower bound for the first eigenvalue of the
Dirac operator with torsion is found that generalizes Friedrich's classical 
Riemannian estimate. We also determine a novel twistor and Killing equation
with torsion and use it to discuss the case in which the minimum is attained 
in the bound.
\end{abstract}
%
\maketitle
%

%
%
\section{Introduction \& summary}
%
\noindent
This paper is devoted to a systematic investigation--via twistor theory--of
the Dirac spectrum of compact Riemannian spin manifolds $(M^n,g)$ with a metric 
connection $\nabla$ with skew-symmetric torsion $T\in\Lambda^3(M^n)$.
The manifolds we consider are non-integrable
geometric structures endowed with the characteristic connection
$\nabla=\nabla^c$ (see the survey \cite{Agricola06}). A.~Gray was
the first to investigate 
manifolds and connections of this kind using the notion of weak holonomy
\cite{Gray71}. Nowadays many different ways exist to tackle the issue of 
weak holonomy, and they can all be described in our setting: to name but a few, 
the intrinsic-torsion approach  (\cite{Salamon89}, \cite{Swann00}), or that 
involving  the critical points of some distinguished functional defined 
on differential forms \cite{Hitchin00}. 
The Dirac operator that one should look at is, hence, not  the
one associated with $\nabla^c$, but rather 
$$\D = D^g + \frac{1}{4}T,$$ 
where $D^g$ is the Riemannian Dirac operator. This generalized Dirac 
operator with torsion  corresponds to the torsion form 
$T/3$ (see  \cite{Bismut}, \cite{Agricola&F04a}, \cite{Agricola&F04b}).
As a matter of fact, $\D$ coincides with the so-called
``cubic Dirac operator'' studied by B.~Kostant (\cite{Kostant99},
\cite{Agricola03}) on naturally reductive spaces, and also with the
Dolbeault operator of a Hermitian manifold (\cite{Bismut}, \cite{Gauduchon97}). 
More recently, theoretical physicists from superstring theory have begun 
to take interest in the operator $\D$ and its symmetries \cite{Houri&KWY10}.
To obtain spectral estimates it turns out crucial to require $\nabla^c T=0$, 
for it is this 
conservation law that ensures the compatibility of
the actions of $\nabla^c$ and $T$ (viewed as an endomorphism) on the spinor
bundle.  
There are several manifolds that are  classically known to admit parallel
characteristic torsion, namely 
nearly K\"ahler manifolds, Sasakian manifolds, nearly parallel $G_2$-manifolds,
and naturally reductive spaces; these classes have been considerably enlarged in
more recent  work (see \cite{Vaisman79}, \cite{Gauduchon&O98},
\cite{Friedrich&I1}, \cite{Alexandrov03}, \cite{Alexandrov&F&S04},
\cite{Fr07a}, \cite{Schoemann}), eventually leading to a host of instances 
to which our results can be applied.

On a Riemannian manifold $(M^n,g)$ Penrose's twistor operator 
($\nabla^g$ denotes  the Levi-Civita connection)
\bdm\tag{$*$}
P \ :=\ \sum_{k=1}^ne_k\ox \left[\nabla^g_{e_k}\psi+\frac{1}{n}e_k\cdot
  D^g\psi  \right]
\edm
is known to encode much information on both the spinorial behaviour and the 
conformal geometry of the underlying manifold  
(see \cite{Lichnerowicz87}, \cite{Lichnerowicz88}, \cite{Hijazi&L88},
\cite{Friedrich89}, \cite{Habermann90}).
The key point to us is that the twistor operator can be used to prove 
Friedrich's estimate for the smallest eigenvalue $\lambda^g$ 
 of $(D^g)^2$ (\cite{Friedrich80}, \cite{Semmelmann98})
\bdm\tag{$**$}
\lambda^g\ \geq\ \frac{n}{4(n-1)}\min_{x\in M^n} \Scalg,
\edm
and to discuss the case where equality holds. 
But whereas Dirac operators with torsion are by now well-established analytical
tools in the study of special geometric structures, all attempts to develop a sort of twistor
theory with torsion have failed, so far. The main problem is that
the operator defined by $(*)$ has no straightforward generalization in presence 
of torsion; one might try to replace $\nabla^g$ by the canonical connection $\nabla^c$, 
or equally well substitute $D^g$ with $\D$ or with the Dirac operator
of $\nabla^c$. Yet one realises quickly that either possibility is unlikely to be 
very meaningful from the geometrical viewpoint.

\medskip
In this article we derive a twistor operator with torsion by asking which
generalization of $P$ yields, on a suitable class of geometries
with torsion, a lower estimate for the  smallest eigenvalue of $\D$ that
contains the optimal estimate $(**)$ in the limiting case of vanishing torsion.

Throughout the article we will assume $(M^n,g)$ is an oriented Riemannian
manifold endowed with a metric connection $\nabla^c$ with skew-symmetric torsion
$T\in\Lambda^3(M^n)$. The situation we have in mind is 
that of the characteristic connection of a $G$ structure;
as described in \cite{Agricola06}, this is--if existent--the
unique $G$-invariant metric connection with skew-symmetric torsion, and is 
well understood in all standard geometries. 
It has to be stressed, however, that our results apply to \emph{any}  metric connection
with parallel skew-symmetric torsion.
 It will be useful to consider the one-parameter
family of connections
\bdm
\nabla^s_X Y \ =\ \nabla^g_X Y + 2s\, T(X,Y,-),
\edm
with normalisation chosen so that $\nabla^s$ has torsion 
$T$ if $s=1/4$, whence $\nabla^c =\nabla^{1/4}$. Obviously 
$\nabla^0=\nabla^g$, so $\nabla^s$ can be thought of as a line in the space of 
connections joining the Levi-Civita to the characteristic connection.
For each $s$ the respective scalar curvatures fulfill $\Scal^s=\Scal^g-24s^2 \|T\|^2$.
The connection $\nabla^s$ may be lifted to the spin bundle $\Sigma M$, and 
will be denoted by the same symbol,
\bdm
\nabla^s_X\psi \ = \ \nabla^g_X\psi+ s(X\haken T)\cdot \psi.
\edm
The spin connection $\nabla^s$ induces a twistor operator $P^s$.
At the heart of the paper lies a twistorial integral equation, 
which is the content of Theorem \ref{twistorial-int-FL}.
This leads to a twistorial eigenvalue estimate for $\D$ that improves
all existing eigenvalue estimates known (Corollaries \ref{cor.twistorial-est-mu}
and \ref{cor.twistorial-est}) and has a wider application range than these.
The limiting case is obtained precisely when the eigenspinor is a twistor
spinor for the twistor operator with torsion $P^s$, with $\displaystyle
s=\frac{n-1}{4(n-3)}$;  thus, its torsion is a multiple of the initial 
characteristic torsion \emph{and depends on the dimension $n$ of $M$}, through
the parameter $s$. This is a rather surprising fact, and it 
explains why it had not been possible to guess the `right' twistor operator beforehand,
although  ($*$) for the Riemannian case  already indicates 
that any answer must involve $\dim M$. In Lemma \ref{lem.twistorglg}, we show that
the twistor equation $P^s\psi =0$ is equivalent to the field equation
\bdm
\nabla^c_X\psi +\frac{1}{n}X\cdot \D\psi + \frac{1}{2(n-3)}
(X\wedge T) \cdot\psi \ =\ 0 \ \ \forall X.
\edm
Friedrich's original proof of his estimate relies on a clever deformation of
the Levi-Civita connection, not on twistor techniques. The same idea was 
later used for the operator $\D$ as well, see
\cite{Agricola&F&K08},
 \cite{Kassuba10}. But in contrast to the
Riemannian situation (first described  in \cite{Semmelmann98} and 
\cite{Hijazi98}), the twistor approach yields 
results that \emph{differ} from the deformation ansatz in the presence of 
torsion. Section \ref{sec.discussion} 
thoroughly discusses the estimate obtained, and compares it to the
other available estimates, if any.

In Section \ref{sec.KTS} we prove that  
twistor spinors with torsion generalize Killing
spinors with torsion (as of Definition \ref{Killing-spinor-wt}) in the most natural way, 
and then we discuss the basic geometric properties of both kinds. We compute
the full integrability condition for the existence of Killing 
spinors with torsion (whose details are deferred to Appendix \ref{app.int-cond}). 
This constraint is then used to prove that Einstein-Sasaki manifolds 
cannot admit Killing spinors with torsion (Corollary \ref{cor.noKSonES}). This result,  
albeit obtained as a by-product of the aforementioned discussion, is remarkable 
in its own right. On the other hand, we show that non-trivial twistor, and even Killing, spinors with
torsion do exist: noteworthy instances are certain $5$- and $7$-dimensional Stiefel manifolds 
endowed with their natural contact structures (Examples \ref{exa.Stiefel} and \ref{exa.Stiefel-2}).

It emerges from the treatise that the case of dimension $n=6$ stands out (Section \ref{section.n6}). 
We prove that, in this distinguished situation, the Killing equation and 
the twistor equation \emph{are equivalent} (Corollary \ref{cor.TSequalsKS}).
For nearly K\"ahler manifolds, we can even prove that the
classes of Riemannian Killing spinors, $\nabla^c$-parallel spinors,
 Killing spinors with torsion, and twistor spinors with torsion coincide
(Theorem \ref{thm.nK-all-spinors}).

In the last part of the paper the twistorial approach is applied to
manifolds with reducible characteristic  holonomy.
It is a standard fact that the splitting of the tangent bundle of a Riemannian 
spin manifold $(M^n,g)$ under the action of the Riemannian holonomy group has 
important consequences for the spectrum of the Dirac operator $D^g$ 
(\cite{ECKim04}, \cite{Alexandrov07}); in the simplest
one-dimensional case, that assumption just means that $M$ admits a 
$\nabla^g$-parallel vector field \cite{Alexandrov&G&I98}. 
So in a similar fashion we can consider local products of
manifolds with parallel characteristic torsion, called, for the present purposes, 
\emph{geometries with reducible parallel torsion}; the precise formulation is found in
Definition \ref{dfn.geom-redpargeom}.
We analyse in detail manifolds with reducible parallel torsion
and their curvature properties, the study of which was lacking in the literature. 
We derive the necessary partial Schr\"odinger-Lichnerowicz formulas compatible with the
splitting (Proposition \ref{partial-SL-1}), and from this obtain another interesting spectral 
estimate for $\D$ (Theorem \ref{thm.twistorial-est-prod}): 
roughly speaking, the estimate is the same as in Corollary \ref{cor.twistorial-est},
but now the dimension $n$ of the manifold is replaced by the largest dimension
of a parallel distribution of the tangent bundle. This result is complemented 
by the 
ensuing discussion of the equality case. 
The section's closing result (Theorem \ref{thm.improvement-BK}) 
gives an interesting eigenvalue inequality for local products in the
Riemannian case.
The proof is again based on twistor techniques and, alas, we show that is has no
analogue  for connections with torsion.
%
\section{Review of the universal eigenvalue estimate}\noindent
%
We recall here the generalized Schr\"odinger-Lichnerowicz identities
for Dirac operators with torsion and the eigenvalue estimates one can derive
from them. A crucial  first order
differential operator that will appear in several instances is
\bdm
\mathcal{D}^s\ =\ \sum_{i=1}^n (e_i\haken T)\cdot \nabla^s_{e_i} .
\edm
Contrary to a Dirac operator with torsion, it has no Riemannian
counter part. Furthermore, define (see Appendix \ref{formulas})
\bdm
\sigma_T\ :=\ \frac{1}{2}\sum_i (e_i\haken T)\wedge (e_i\haken T).
\edm
In \cite{Friedrich&I1}, the following identities are proved:
\begin{thm}\label{FI}\label{thm.Fr&I}
%
\begin{enumerate}
\item[]
\item The square of the Dirac operator $D^s$ satisfies the relation
\bdm
(D^s)^2\ =\ \Delta^s+3s\, dT -8s^2\sigma_T+2s\, \delta T - 4s\, \mathcal{D}^s
+\frac{1}{4}\Scal^s.
\edm
\item The anticommutator of  $D^s$ with $T$ is given by
\bdm
D^s T+T D^s\ =\ dT +\delta T - 8s\sigma_T -2\mathcal{D}^s.
\edm
\end{enumerate}
\end{thm}
Since the spectrum of $\mathcal{D}^s$ is usually beyond control,
these relations are hard to evaluate for, say, non parallel spinor fields.
The following improvement was first proved by Bismut in 
\cite[Thm 1.10]{Bismut}, see also \cite[Thm 6.2]{Agricola&F04a}.
\begin{thm}[Generalized Schr\"odinger-Lichnerowicz formula]\label{thm.gSLF}
For arbitrary torsion $T$, one has the identity
\bdm\tag{$*$}
(D^{s/3})^2\ =\ \Delta^s+ s\, dT +\frac{1}{4}\Scal^g-2s^2\|T\|^2.
\edm
If in addition  $dT=2\sigma_T$, this may be simplified to
\bdm
(D^{s/3})^2\ =\ \Delta^s - s T^2+\frac{1}{4}\Scal^g+ (s-2s^2)\, \|T\|^2.
\edm
\end{thm}
This is in particular satisfied if the torsion is parallel for  $s=1/4$.
In this case, the last relation has a remarquable
consequence. For then $\Delta^s$ commutes with $T$, and this is
trivially correct for the multiplication by $T^2$ and  by scalars,
hence (compare \cite[Prop. 3.4]{Agricola&F04b})
\bdm
(D^{s/3})^2\circ T \ =\ T\circ (D^{s/3})^2.
\edm
It is therefore possible to split the spin bundle in the orthogonal sum
of its eigenbundles for the $T$ action,
\bdm
\Sigma M\ =\ \bigoplus_{\mu}\Sigma_{\mu},
\edm
and to consider $(D^{s/3})^2$ on each of them, since
$\nabla^s$ and $(D^{s/3})^2$ both preserve this splitting.
We shall henceforth denote the different eigenvalues of $T$ on $\Sigma M$ by
$\mu_1,\ldots,\mu_k$. We therefore obtain the following 
\emph{universal eigenvalue estimate}
for the first eigenvalue $\lambda=\lambda(\D^2)$ of  $\D^2$ for the
connection with torsion $T/3$.
We state it separately on $\Sigma_\mu$ and the whole spin bundle
$\Sigma M$, for examples teach us that going over to $\Sigma M$
often means throwing away too much detail information.
\begin{thm}[Universal eigenvalue estimate]\label{thm.universal-est}
For $\nabla^c T=0$, the smallest eigenvalue $\lambda$ of $\D^2$
on $\Sigma_\mu$ satisfies the inequality
\bdm
\lambda\big(\D^2\big|_{\Sigma_\mu}\big)\ \geq\
\frac{1}{4}\Scalgmin + \frac{1}{8}\|T\|^2 -\frac{1}{4}\, \mu^2\ =:\
\beta_{\univ}(\mu),
\edm
and equality occurs if and only if $\Scalg$ is constant and
$\Sigma_\mu$ contains a $\nabla^c$-parallel spinor.
For the smallest eigenvalue $\lambda$ of $\D^2$ on the whole
spin bundle $\Sigma$, one thus obtains the estimate
\bdm
\lambda\ \geq\ \frac{1}{4}\Scal^g + \frac{1}{8}\|T\|^2 -\frac{1}{4}\,
\max(\mu_1^2,\ldots,\mu_k^2)\ =:\ \beta_{\univ}.
\edm
\end{thm}
Equality occurs if and only if the scalar curvature is constant,
and the eigenspinor is $\nabla^c$-parallel and a $T^2$-eigenspinnor
with eigenvalue $\max(\mu_1^2,\ldots,\mu_k^2)$;
this can indeed happen in some special geometries (compare 
Section \ref{subsection.par-spin}). If $T=0$, this is an estimate
for the Riemannian Dirac operator that fails to
be optimal. It was improved 1980 by Thomas Friedrich
by a clever deformation trick for the Levi-Civita connection.
This method, which we will (somehow vaguely) call the \emph{deformation
method} in this paper, was successfully applied to Dirac operators with
torsion (\cite{Agricola&F&K08}, 
\cite{Kassuba10}. Nevertheless,
an optimal estimate could not be derived in all cases of interest and many
open questions remain.

An alternative approach to Friedrich's inequality is by
twistorial techniques. One goal of this article is thus to
work out this ansatz in detail for connections with torsion,
and to improve the results obtaind by the deformation method.
Contrary to the Riemannian case, the two approaches turn out not to be
equivalent.
%
%
\section{The  twistorial eigenvalue estimate}\noindent
%
If $m: TM\ox \Sigma M\ra\Sigma M$ denotes Clifford multiplication,
the projection $p: TM\ox \Sigma M\ra\ker m \subset TM\ox \Sigma M$
is locally given by
\bdm
p(X\ox \psi)\  =\ X\ox \psi+\frac{1}{n}\sum_{k=1}^n e_k\ox e_k
\cdot X\cdot\psi.
\edm
The \emph{Penrose-} or \emph{Twistor operator} is the composition
$P^s :=p\circ\nabla^s$. Locally,
\bdm
P^s\psi=\sum_{k=1}^ne_k\ox \{\nabla^s_{e_k}\psi+\frac{1}{n}e_k\cdot D^s\psi \}\,.
\edm
A spinor  $\psi$ is called a  \emph{twistor spinor} if it lies in the
kernel of $P^s$: $P^s\psi=0$. This is equivalent to the
\emph{twistor equation} (for still arbitrary parameter value $s$)
\bdm
\nabla^s_X\psi +\frac{1}{n}X\cdot D^s\psi\ =\ 0,
\edm
which has to hold for any vector field $X$.
One easily checks that some properties of Riemannian twistor spinors
(\cite[Section 1.4, Thms 2, 3]{BFGK})
carry over without modification to the case with torsion.
We omit the proof.
\begin{thm}\label{thm.TS-props}
\begin{enumerate}
\item[]
\item $\psi$ is a twistor spinor if and only if the following
condition holds for
any vector fields $X,Y$:
\bdm
X\nabla^s_Y\psi + Y\nabla^s_X\psi \ =\ \frac{2}{n}g(X,Y)D^s\psi.
\edm
\item $\psi$ is a twistor spinor if and only if the expression
 $X\cdot\nabla^s_X\psi$ does not depend on the unit vector field $X$.
\item Any  twistor spinor $\psi$ satisfies: \ $(D^s)^2\psi = n\,
\Delta^s \psi$.
\item Any spinor field $\vphi$ satisfies:
\ $\displaystyle \|P^s\vphi\|^2 +\frac{1}{n}\|D^s\vphi\|^2 =
\|\nabla^s \vphi\|^2 $.
\end{enumerate}
\end{thm}
The following calculation is fundamental for the twistorial estimate.
\begin{lem}\label{lem.fund-calc}
\bea[*]
(D^{s/3})^2 - \frac{1}{n} (D^s)^2 & = &
\frac{n-1}{n} \left[ D^0 + \frac{s(n-3)}{n-1} T\right]^2
+ \frac{4s^2}{1-n}T^2\\
& = & \frac{n-1}{n} \left[ D^s - \frac{2sn}{n-1} T\right]^2 + \frac{4s^2}{1-n}T^2.
\eea[*]
\end{lem}
\begin{proof}
Consider the difference ($D^0=D^g$):
\bea[*]
(D^{s/3})^2-\frac{1}{n}(D^s)^2 &=& (D^0+sT)^2-\frac{1}{n}(D^0+3sT)^2\\
&=&(D^0)^2+s^2T^2+s(D^0T+TD^0)-\frac{1}{n}[(D^0)^2+9s^2T^2+3s(D^0T+TD^0)]\\
&=&\left(1-\frac{1}{n}\right)(D^0)^2+s^2\left(1-\frac{9}{n}\right)T^2
+s\left(1-\frac{3}{n}\right)(D^0T+TD^0)\\
&=& \frac{n-1}{n}[(D^0)^2+s^2\frac{n-9}{n-1}T^2+s\frac{n-3}{n-1}(D^0T+TD^0)].
\eea[*]
The square of any Dirac operator $D^0+\mu T$ can be expanded into
\bdm
(D^0+\mu T)^2=(D^0)^2+\mu(D^0T+TD^0)+\mu^2T^2\,.
\edm
If we set $\mu=s\frac{n-3}{n-1}$, the difference above may be
rewritten as
\bea[*]
\qquad (D^{s/3})^2-\frac{1}{n}(D^s)^2&=&
\frac{n-1}{n}\left[\big(D^0+\mu T\big)^2-s^2\frac{4n}{(n-1)^2}T^2\right]\\
&=&\frac{n-1}{n}\left[\left(D^0+s\frac{n-3}{n-1} T\right)^2-
s^2\frac{4n}{(n-1)^2}T^2\right]\\
&=&\frac{n-1}{n}\left[D^0+s\frac{n-3}{n-1} T\right]^2-s^2\frac{4}{n-1}T^2
\qquad \qedhere
\eea[*]
\end{proof}
This calculation allows us to prove a crucial integral formula, which will
yield the desired estimate as an easy corollary. The key idea is
that the difference of squares of Dirac operators on the left hand side
can again by expressed as the square of a suitably renormalized
Dirac operator (and a multiple of the endomorphism $T^2$).
We may assume without loss that $n\geq 4$, since the case $n=3$ is not
very interesting. Recall that we  write
$\nabla^c$ for the connection with parameter $s=1/4$.
\begin{thm}[Twistorial integral formula]\label{twistorial-int-FL}
Suppose $\nabla^c T=0$. For any spinor field $\psi$,
the Dirac operator $\D$ of the connection with torsion
$\frac{1}{3}T$ satisfies the following integral formula:
\bea[*]
\int_M \langle \D^2\psi,\psi\rangle dM& =&
\frac{n}{n-1}\int_M  \|P^s\psi\|^2 dM+ \frac{n}{4(n-1)}\int_M
\Scal^g\|\psi\|^2dM\\
&& +\frac{n(n-5)}{8(n-3)^2}\|T\|^2\int\|\psi\|^2 dM
+\frac{n(4-n)}{4(n-3)^2}\int_M \langle T^2\psi,\psi\rangle dM.
\eea[*]
Here, the parameter $s$ appearing in the twistor operator $P^s$ has the
value  $\displaystyle s=\frac{n-1}{4(n-3)}$.
\end{thm}
\begin{proof}
Consider the operator $D^{s/3}=D^g+sT$. Integrating over the generalized
Schr\"odinger-Lichnerowicz formula ($*$) of Theorem \ref{thm.gSLF},
we obtain (we omit the volume form and the domain
of integration in most integrals)
\bdm
\int\langle (D^{s/3})^2\psi,\psi\rangle \ =\
\int \|\nabla^s\psi\|^2 + s\int \langle dT\psi,\psi\rangle +\frac{1}{4}
\int \Scal^g\|\psi\|^2 -2s^2\int \|T\|^2\|\psi\|^2.
\edm
By identity (4) from Theorem \ref{thm.TS-props}, the length
$\|\nabla^s\psi\|^2$ can be expressed through the twistor and
the Dirac operator, yielding
\bdm
 \int \langle [(D^{s/3})^2-\frac{1}{n}(D^s)^2] \psi,\psi\rangle
=\int\|P^s\psi\|^2 + s \int\langle dT\psi,\psi\rangle
+\frac{1}{4}\int \Scal^g\, \|\psi\|^2
-2s^2 \int \|T\|^2\|\psi\|^2 .
\edm
The main idea is now to view the difference on the left hand side
as the square of a single Dirac operator by a clever choice of the
parameter $s$. We rewrite the left hand side using the fundamental
calculation from Lemma \ref{lem.fund-calc},
\bea[*]
\frac{n-1}{n}\int \langle  (D^0+s\frac{n-3}{n-1}T)^2 \psi,\psi\rangle & =&
\int\|P^s\psi\|^2+\frac{1}{4}\int \Scal^g \|\psi\|^2+\\
&& +s\int\langle dT\psi,\psi\rangle -2s^2 \int \|T\|^2\|\psi\|^2
+s^2\frac{4}{n-1}\int \langle T^2\psi,\psi\rangle .
\eea[*]
We now choose the parameter $s$ such that the operator on the left hand
side becomes just $\D$, i.\,e.~the Dirac operator with
torsion $\frac{1}{3}T$. Since $\D=D^0+\frac{1}{4}T$,
this requires  $s\frac{n-3}{n-1}=\frac{1}{4}$, hence we obtain
$s=\frac{n-1}{4(n-3)}$, the value encountered in the statement of
the result. Inserting this value of $s$ yields
\bea[*]
 \frac{n-1}{n}\int \langle  \D^2 \psi,\psi\rangle
&=& \int\|P^s\psi\|^2+\frac{1}{4}\int \Scal^g \|\psi\|^2+\\
&&+\frac{n-1}{4(n-3)}\cdot\int\langle dT\cdot\psi,\psi\rangle
-\frac{(n-1)^2}{8(n-3)^2} \int \|T\|^2\|\psi\|^2+\frac{n-1}{4(n-3)^2}
\int \langle T^2\psi,\psi\rangle.
\eea[*]
The assumption $\nabla^c T=0$ implies $dT=2\sigma_T$,
and since
$T^2=-2\sigma_T+\|T\|^2$ always holds, we get  $dT=-T^2+\|T\|^2$.
This means for the previous equation
\bdm
\int \langle  \D^2 \psi,\psi\rangle  =
\frac{n}{n-1}\int\|P^s\psi\|^2+\frac{n}{4(n-1)}\int \Scal^g\|\psi\|^2
+\frac{n(n-5)}{8(n-3)^2}\int \|T\|^2\|\psi\|^2
-\frac{n(n-4)}{4(n-3)^2} \int \langle T^2\psi,\psi\rangle .
\qedhere
\edm
\end{proof}
It is to be understood that all eigenvalue estimates based on this
integral identity are meant on \emph{compact} manifolds, even if this
is not repeated throughout. We can assume that an eigenspinor
 $\psi$ of  $\D^2$ with eigenvalue
$\lambda$ lies in one subbundle $\Sigma_\mu$ (see the general comments
on the universal estimate). Thus, we obtain:
\begin{cor}[Twistorial eigenvalue estimate in $\Sigma_\mu$]
\label{cor.twistorial-est-mu}
For $\nabla^c T=0$, the smallest eigenvalue $\lambda$ of $\D^2$
on $\Sigma_\mu$ satisfies the inequality
\bdm
\lambda\big(\D^2\big|_{\Sigma_\mu}\big) \ \geq\
\frac{n}{4(n-1)}\Scal^g_{\min}+ \frac{n(n-5)}{8(n-3)^2}\|T\|^2
+\frac{n(4-n)}{4(n-3)^2}\mu^2\ =:\  \beta_{\tw} (\mu),
\edm
and equality holds if and only if the following conditions are
satisfied:
\begin{enumerate}
\item the Riemannian scalar curvature of $(M,g)$ is constant,
\item the eigenspinor $\psi$  is a twistor spinor for
 $\displaystyle s=\frac{n-1}{4(n-3)}$.
\end{enumerate}
\end{cor}
\begin{cor}[Twistorial eigenvalue estimate]\label{cor.twistorial-est}
For $\nabla^c T=0$, the smallest eigenvalue $\lambda$ of $\D^2$
satisfies the inequality
\bdm
\lambda \ \geq\ \frac{n}{4(n-1)}\Scal^g_{\min}+ \frac{n(n-5)}{8(n-3)^2}\|T\|^2
+\frac{n(4-n)}{4(n-3)^2}\max(\mu_1^2,\ldots,\mu_k^2)\ =:\  \beta_{\tw},
\edm
and equality holds if and only if, in additon to the two conditions
from the previous Corollary,  the eigenspinor  $\psi$ lies in the subbundle
$\Sigma_\mu$ corresponding to the largest eigenvalue of $T^2$.
\end{cor}
We make some first pertinent comments on this result.
\begin{enumerate}
\item The twistorial eigenvalue estimate reduces for  $T=0$
to Friedrich's inequality, thus showing its optimality at least in this
situation. The quality of the estimate increases if the scalar curvature
becomes large and dominates the terms in  $\|T\|^2$ and 
$\max(\mu_1^2,\ldots,\mu_k^2)$. But there also exist
a few examples with
negative scalar curvature (for example, Sasaki metrics on
compact quotients of Heisenberg groups, see \cite{Friedrich&I1},
\cite{Friedrich&I2}, \cite{Agricola&F04b}); for these manifolds,
the twistorial eigenvalue estimate becomes rather bad.
\item
In some situations, the geometric data yield additional information
in which
subbundle $\Sigma_\mu$ the smallest eigenvalue can occur, or which bundles are
of particular interest.
In these cases, the global estimate from Corollary \ref{cor.twistorial-est}
is too coarse, and one should apply instead the estimate in one or several
well-chosen subbundles  as stated in Corollary \ref{cor.twistorial-est-mu}.
The case of parallel spinors discussed in Section \ref{subsection.par-spin}
below is an example of such a situation.
\item Strictly speaking, there is not one equality case, but one in
every subbundle $\Sigma_\mu$. Thus, it may happen that the
twistorial eigenvalue estimate on the whole spin bundle $\Sigma$ is not
sharp (for the $\mu$ belonging to
the maximum of the $T$ eigenvalues), but twistor spinors exist
nevertheless (namely, for some other $\mu$).
\item Unfortunately, one cannot construct a $\D$ eigenspinor from a
 $\D^2$ eigenspinor that would still lie in one of the
subbundles $\Sigma_\mu$ -- hence, a twistor spinor realizing the optimal
eigenvalue estimate does not have to be some kind of Killing spinor.
Nevertheless, a very reasonable Killing equation with torsion exists, and
every Killing spinor with torsion is necessarily a twistor spinor
with torsion. In dimension $6$, the converse can be shown
(see Section \ref{section.n6}).
\end{enumerate}
In the next Section, we will discuss applications of this estimate in
different
special geometries with torsion and compare it to the universal
estimate. Section  \ref{sec.KTS} will be devoted to the discussion of the
equality case in the twistorial estimate, in particular to the
description of Killing and twistor spinors with torsion and
examples of manifolds where such spinors exist.

\section{Discussion of the twistorial estimate}\label{sec.discussion}
%
\subsection*{The case $n=4$}
The $4$-dimensional case is special in many respects.
For purely algebraic reasons, $\sigma_T=0$, hence $\nabla^c T=0$ implies
$dT=0$ and $T^2$ acts by scalar multiplication with $\|T\|^2$, i.\,e.
~the only $T$ eigenvalues are $\pm \|T\|$.
Furthermore, $\nabla^c T=0$ implies $\nabla^g *T=0$, i.\,e.~there
exists a LC-parallel $1$-form on $(M^4,g)$. Set $c=\Scalgmin/\|T\|^2$.
In \cite{Agricola&F&K08} it was proved that
\bdm
\lambda \ \geq\ \left\{\ba{ll} \frac{\|T\|^2}{4}\left[c-\frac{1}{2}\right]
 & \text{ for } c\geq 3/2,\\
\frac{\|T\|^2}{16} [\sqrt{6c} - 1]^2 & \text{ for } 1/6\leq c\leq 3/2
\ea\right.
\edm
The first estimate is just the universal estimate given in Theorem
\ref{thm.universal-est}. Indeed, the deformation method used in this paper
has the typical property of yielding eigenvalue estimates that are valid
only for some restricted parameter range. In particular, no improvement
was possible for $c\geq 3/2$. In contrast, the twistorial eigenvalue
estimate from Corollary \ref{cor.twistorial-est} yields
\bdm
\lambda \ \geq \ \frac{\|T\|^2}{3}\left[c-\frac{3}{2}\right] \ \ \text{ for }
c\geq 3/2.
\edm
Hence, the parameter range for which the  twistor ansatz yields an
improvement  is complementary
to the results obtained via deformation techniques.
One checks that the twistor estimate lies above the universal estimate
for $c\geq 9/2$.

We observe that for $n=4$, eigenvalue estimates in terms of the infimum of 
the conformal scalar curvature are available without the assumption
of parallel torsion, see \cite{Alexandrov&I00},
\cite{Dalakov&I01}. Unfortunately, these different curvature quantities cannot
easily be compared (see the more detailed comments on this point in 
\cite{Agricola&F&K08}).
\subsection*{The case $n=5$}
For a $5$-dimensional manifold, the twistorial eigenvalue estimate
becomes
\bdm
\lambda\ \geq\ \frac{5}{16}[\Scalg_{\min}  -\max (\mu^2_1,\ldots, \mu_k^2)].
\edm
Thus, the quality of the estimate increases for large scalar curvatures.
In Example \ref{exa.Stiefel}, we show that the Stiefel manifold
$V_{4,2}=\SO(4)/\SO(2)$ carries a metric for which this estimate becomes
optimal. On the other hand, we can identify manifolds for which the
twistorial estimate yields no improvement. This is for example the
case for Sasaki manifolds $(M^5,g,\xi,\eta,\varphi)$: In this case, 
there exists a unique connection $\nabla$
with totally
skew-symmetric torsion preserving the Sasakian structure by
\cite{Friedrich&I1}. The torsion form is given by the formula
$T = \eta \wedge d \eta$, and $||T||=8$ holds. $T$ splits the 
spinor bundle into two
$1$-dimensional bundles and one $2$-dimensional bundle,
\bdm
\Sigma_{\pm4} \ = \ \big\{ \psi \in \Sigma M^5 \, : \, T\psi \, = \, \pm 
\, 4 \, \psi \big\} \, , \quad
\Sigma_{0} \ = \ \big\{ \psi \in \Sigma M^5 \, : \, T\psi \, = \, 0 \big\} 
\, .
\edm
Thus, $\max (\mu^2_1,\ldots, \mu_k^2)=16$ and the twistorial estimate
becomes 
\bdm
\lambda\ \geq\ \frac{5}{16}\Scalg_{\min} - 5.
\edm
On the other hand, it was proved by the deformation ansatz in
\cite{Agricola&F&K08}
that
\bdm
\lambda \ \geq\ \left\{\ba{ll} 
\frac{1}{16}\left[1+\frac{1}{4}\Scalg_{\min}\right]^2 & \text{ for } 
-4<\Scalg_{\min}\leq 4(9+4\sqrt{5})\\[2mm]
\frac{5}{16}\, \Scalg_{\min} & \text{ for }\ \Scalg_{\min}\geq 4(9+4\sqrt{5})
\simeq 71,78.
\ea\ \right.
\edm
The estimates coincide for $\Scalg_{\min}=36$; for all other possible
scalar curvatures, the deformation estimate is better. Thus, the main
advantage of the twistorial estimates lies here in its universality:
It makes a statement for non-Sasaki manifolds as well, a case that is
not covered by \cite{Agricola&F&K08}.
\subsection*{The case $n=6$}
Nearly K\"ahler manifolds will be discussed in Section \ref{section.n6}.
Hence, let us consider some of the other classes of manifolds with parallel
characteristic torsion. 
 Almost Hermitian $6$-manifolds with parallel characteristic
torsion were classified by Schoemann in \cite{Schoemann}; in particular, it
was shown that there exist many almost Hermitian manifolds of Gray-Hervella type
$W_3$ or $W_4$ with  parallel characteristic torsion -- nilpotent Lie groups,
naturally reductive spaces, $S^1$-fibrations over Sasaki $5$-manifolds etc.
For both classes, the torsion has eigenvalues $\mu=0,\pm \sqrt{2}\,\|T\|$,
thus the twistorial eigenvalue estimate for $W_3$ or $W_4$ geometries with
parallel torsion is given by
\bdm
\lambda\ \geq\ \frac{3}{10}\Scalg_{\min}  -\frac{7}{16}\|T\|^2.
\edm
The right hand side is non-negative for $\Scalg_{\min}\geq 35\|T\|^2/18$.
For this curvature range, the deformation technique did not yield
any improvement of the universal eigenvalue estimate \cite{Kassuba10}. 
However, it was proved
therein that there exist no $\nabla^c$-parallel spinors, hence 
the universal estimate
could not be optimal; thus, the twistor estimate is better for these
large scalar curvatures.
\subsection*{Existence of $\nabla^c$-parallel spinors}\label{subsection.par-spin}
In general, the twistor and the universal estimate cannot be compared
abstractly. But
if there exists a $\nabla^c$-parallel spinor field $\psi\in\Sigma_\mu$,
the universal eigenvalue estimate $\lambda$
(see Theorem \ref{thm.universal-est}) is sharp for some $T$ eigenvalue $\mu$,
i.\,e.
\bdm
\lambda\ =\ \frac{1}{4}\Scalgmin +\frac{1}{8}\|T\|^2 -\frac{1}{4}\mu^2\
= :\ \beta_\univ (\mu).
\edm
Notice that $\Scalg$ has to be constant in this situation: By  identity (1)
from Theorem \ref{thm.Fr&I} \cite[Cor.~3.2]{Friedrich&I1} for  parallel torsion, 
such a spinor satisfies $\sigma_T\psi+ \Scal^c \psi / 4 = 0$, so the fact that
$\nabla^c\sigma_T=0$ implies $\Scal^c=\mathrm{const}$, and then
the claim follows (this generalizes the well-known fact that the Riemannian
scalar curvature
vanishes in the presence of a $\nabla^g$-parallel spinor, see
\cite{Hitchin74}).  Hence we can drop
the minimum in the formula for $\lambda$.
On the other hand, $\nabla^c\psi=0$ implies
$\D\psi =-\frac{1}{2}T\psi=-\frac{\mu}{2}\psi$, hence $\lambda=\mu^2/4$.
Thus, we have the relation
\bdm\tag{$*$}
\Scalg\ =\  -\frac{1}{2}\|T\|^2+2\mu^2.
\edm
A priori, it is not so easy to compare this result with the twistorial
eigenvalue estimate $\lambda\geq \beta_\tw(\mu)$. However, 
in the presence of  $\nabla^c$-parallel spinor fields, the twistorial 
estimate cannot be larger than the universal estimate, 
i.\,e.~$\beta_\tw(\mu)\leq \beta_\univ (\mu)$
needs to hold. This observation leads to the following result, which
is of interest on his own:
\begin{lem}
Suppose that $\nabla^c T=0$,  that there exists at least one
$\nabla^c$-parallel spinor field $0\neq \psi\in\Sigma_\mu$, and that $n\leq 8$.
Then the following inequalities hold:
\bdm\tag{$**$}
0\ \leq\ 2n\|T\|^2 + (n-9)\mu^2, \quad
\Scalg \ \leq\ \frac{9(n-1)}{2(9-n)}\|T\|^2.
\edm
Furthermore, equality is attained if and only if
$\beta_\tw(\mu)= \beta_\univ (\mu)$.
\end{lem}
\begin{proof}
We only sketch the argument, leaving out the routine computations.
First, one checks that $\beta_\tw(\mu)\leq \beta_\univ (\mu)$ is equivalent
to
\bdm
\Scalg \ \leq\ \frac{(n-1)(9-n)}{2(n-3)^2}\|T\|^2 +
\frac{(n-1)(2n-9)}{(n-3)^2}\mu^2.
\edm
Rewriting relation ($*$) as $-\Scalg =\|T\|^2/2-2\mu^2$ and adding this to the
previous inequality, one obtains the first of the two statements.
It becomes trivial for $n\geq 9$. For $n\leq 8$, it can be rewritten
as $\mu^2\leq 2n\|T\|^2/(9-n)$. We then use this in order to eliminate
$\mu^2$ from the identity ($*$), yielding the second statement.
\end{proof}
We thus have an easy criterion for excluding the existence of
parallel spinors and for checking whether the two eigenvalue estimates
yield the same result.

For example, consider a $6$-dimensional nearly K\"ahler manifold
$(M^6,g,J)$ with its characteristic connection $\nabla^c$. These are
Einstein spaces of positive scalar curvature,
$\|T\|^2 = \frac{2}{15}\Scalg$, and $T$ has the eigenvalues
$\mu=0$ (multiplicity $6$) and $\mu=\pm 2\|T\|$ (each with multiplicity $1$).
It is well-known that the two Riemannian
Killing spinors $\vphi_{\pm}$ are $\nabla^c$-parallel and lie
in $\Sigma_{\pm 2\|T\|}$.
One then checks by hand that
 $\beta_\tw(\mu)= \beta_\univ (\mu)=\frac{2}{15}\Scalg$, and indeed
one sees that the relations $(**)$ hold with an equality sign.
Qualitatively, the same happens for nearly parallel $G_2$ manifolds.
%
%
\section{Killing and twistor spinors with torsion}\label{sec.KTS}

\begin{lem}\label{lem.twistorglg}
Suppose $\nabla^c T=0$.
The twistor equation $P^s\psi =0$ corresponding to the
parameter value $\displaystyle s=\frac{n-1}{4(n-3)}$
is equivalent to
\bdm
\nabla^c_X\psi +\frac{1}{n}X\cdot \D\psi + \frac{1}{2(n-3)}
(X\wedge T) \cdot\psi \ =\ 0,
\edm
and each such twistor spinor satisfies
\bdm
\D^2\psi\ =\ \left[\frac{n}{4(n-1)}\Scal^g+\frac{n(n-5)}{8(n-3)^2}
\|T\|^2 + \frac{n(4-n)}{4(n-3)^2}T^2\right]\psi.
\edm
\end{lem}
\begin{proof}
For  $s=\frac{n-1}{4(n-3)}$, one has
\bdm
\nabla^s_X\psi=\nabla^c_X\psi+\frac{1}{2(n-3)}(X\haken T)\cdot \psi
\edm
and
\bdm
D^s=\D+\frac{n}{2(n-3)}T\,.
\edm
Inserting these expressions into the twistor equation leads to
\bdm
\nabla^c_X\psi+\frac{1}{n}X\cdot \D\psi +\frac{1}{2(n-3)}(X\cdot T +
X\haken T)\cdot \psi=0 .
\edm
The claim then follows, since
$X\cdot T=X\wedge T -X\haken T$.
To derive the identity for $\D^2$ on twistor spinors, we proceed
similarly as in the proof
of the twistorial integral formula (Theorem \ref{twistorial-int-FL}), but with
one crucial change. Let $\psi$ be a twistor spinor,  $s=\frac{n-1}{4(n-3)}$.
We start with the generalized
Schr\"odinger-Lichnerowicz formula ($*$) of Theorem \ref{thm.gSLF},
\bdm
 (D^{s/3})^2\psi \ =\
\Delta^s\psi  + s  dT\psi  +\frac{1}{4}\Scal^g\psi -2s^2 \|T\|^2\psi.
\edm
Instead of the  norm
identity  (4) for any spinor from Theorem \ref{thm.TS-props}, we can now use
the operator identity (3), $(D^s)^2\psi=n\Delta^s\psi$ to rewrite this
as
\bdm
 (D^{s/3})^2\psi - \frac{1}{n}(D^s)^2\psi \ =\
s \, dT\psi  +\frac{1}{4}\Scal^g\psi -2s^2 \|T\|^2\psi.
\edm
By the fundamental Lemma \ref{lem.fund-calc}, the left hand side can
be expressed through $\D^2$,
\bdm
\frac{n-1}{n}\D^2\psi - s^2\frac{4}{n-1}T^2\psi \ =\
s\,  dT\psi  +\frac{1}{4}\Scal^g\psi -2s^2 \|T\|^2\psi.
\edm
Now one finishes the calculation as in the proof of Theorem
\ref{twistorial-int-FL},
and it comes as no surprise that the identity obtained is exactly
the operator version of the twistorial integral formula without
$\int \|P^s\psi\|^2\, dM$ term.
\end{proof}
Observe that being  \emph{only} a twistor spinor does not imply
being a $\D^2$ eigenspinor, hence we cannot conclude from the last identity
that the scalar curvature has to be constant if such a spinor exists.
We shall now define the \emph{Killing equation with torsion}
for any dimension $n$.
\begin{dfn}[Killing spinor with torsion]\label{Killing-spinor-wt}
A spinor field $\psi$ is called a \emph{Killing spinor with torsion}
if  the equation $\nabla^s_X\psi = \kappa X\cdot \psi$
for $s=\frac{n-1}{4(n-3)}$ holds. The spinor then satisfies 
$D^s\psi = -n\kappa\psi$, which is equivalent to
\bdm
\D \psi\ =\ -n\kappa\psi - \frac{n}{2(n-3)}T\psi.
\edm
\end{dfn}
Thus, contrary to a twistor spinor with torsion, \emph{any}
Killing spinor with torsion known to lie in some $T$-eigenspace
$\Sigma_\mu$ is a $\D$-eigenspinor.

Observe that the Killing vector fields of a metric connection with
antisymmetric torsion are precisely those of the Levi-Civita connection,
hence the vector field
\be\label{Killing-VF}
X_\psi \ :=\ \sum_{j=1}^n i\langle \psi, e_j\cdot\psi \rangle e_j
\ee
is Killing as in the Riemannian situation (compare \cite[p.~30]{BFGK}).
The next result claims that all Killing spinors are twistor
spinors with matching parameters. We omit the easy proof.
\begin{lem}\label{gen-Killing-eq}
Suppose $\nabla^c T=0$ and $\psi\in\Sigma_\mu$.
The Killing equation $\nabla^s_X\psi = \kappa X\cdot \psi$
for $s=\frac{n-1}{4(n-3)}$ is then equivalent to
\bdm
\nabla^c_X\psi -\left[\kappa +\frac{\mu}{2(n-3)}\right]X\cdot \psi
+\frac{1}{2(n-3)}(X\wedge T)\psi \ =\ 0.
\edm
In particular, $\psi$ is a twistor spinor with torsion
for the same value $s$,  and the
Killing number $\kappa$ satisfies the quadratic equation
\bdm
n \left[\kappa +\frac{\mu}{2(n-3)}\right]^2 \ =\
\frac{1}{4(n-1)}\Scal^g+\frac{n-5}{8(n-3)^2}
\|T\|^2 - \frac{n-4}{4(n-3)^2}\mu^2.
\edm
In particular, the scalar curvature has to be constant.
\end{lem}
\begin{NB}
If $T=0$, and a fortiori $\mu=0$, this quadratic equation reduces
to the well-known relation $\Scal^g = 4n(n-1)\kappa^2$ for Riemannian Killing
spinors \cite{Friedrich80}. However,  $\kappa=0$ does \emph{not} correspond to
$\nabla^c$-parallel spinors because of the torsion shift hidden in the value
of $s$, hence $0$ is a priori an admissible Killing number (contrary to the
Riemannian case). An 
easy formal calculation shows: \emph{Any spinor field  parallel for
the connection with torsion $T$ is a Killing spinor with torsion with $\kappa
=0$ for the connection with torsion $\frac{n-3}{n-1}\,T$}. This can be a
useful remark when parallel spinors are known to exist for a connection
with non-parallel torsion for which the rescaled torsion $\frac{n-3}{n-1}\,T$
becomes parallel. 
\end{NB}
The Killing equation with torsion can be used to express the
curvature operator of a manifold admitting such spinor fields.
Thus, we obtain an algebraic identity for the Ricci tensor;
the rather lengthy proof is deferred to Appendix \ref{app.int-cond}.
\begin{thm}[integrability condition]
%
Let $\psi$ be a Killing spinor with torsion
with Killing number $\kappa$,   set $\lambda:=\frac{1}{2(n-3)}$
for convenience, and recall that $s=\frac{n-1}{4(n-3)}$.
Then the Ricci curvature of the characteristic connection satisfies
the identity
\bea[*]
{\Ric}^c(X)\psi&=&-16s\kappa(X\haken
T)\psi+4(n-1)\kappa^2X\psi+(1-12\lambda^2)(X\haken\sigma_T)\psi+\\
                        & &+2(2\lambda^2+\lambda)\sum e_k(T(X,e_k)\haken T)\psi \,.
\eea[*]
\end{thm}
As a typical application of this result, it is shown in Corollary
\ref{cor.noKSonES}:
\begin{cor}
A $5$-dimensional Einstein-Sasaki manifold $(M,g,\xi,\eta,\vphi)$
endowed with its characteristic
connection cannot admit Killing spinors with torsion.
\end{cor}

\begin{exa}[A $5$-dimensional manifold with Killing spinors with torsion]
\label{exa.Stiefel}
The $5$-dimensional Stiefel manifold $V_{4,2}=\SO(4)/\SO(2)$ carries
a one-parameter family of metrics constructed by G.\ Jensen \cite{Jensen75}
with many remarkable properties. Embed $H=\SO(2)$ into $G=\SO(4)$ as the
lower diagonal $2\x 2$ block. Then the Lie algebra $\so(4)$ splits
into $\so(2)\oplus \m$, where $\m$ is given by
\bdm
\m \ =\left\{ \left[ \ba{c|c} {\ba{cr} 0 & -a \\ a & 0 \ea} & -X^t \\ \hline
X & {\ba{cc} 0 & 0 \\ 0 & 0 \ea} \ea\right] =: (a,X)\, : \
a\in \R,\ X\in \M_{2,2}(\R)\,\right\}\,.
\edm
Denote by $\beta(X,Y):= \tr(X^tY)$ the Killing form of $\so(4)$.
Then the Jensen metric on $\m$ to the parameter $t\in\R$ is  defined
by
\bdm
\lan (a,X), (b,Y)\ran\ =\ \frac{1}{2}\beta(X,Y) + t \beta(a,b)\ =\
\frac{1}{2}\beta(X,Y)+  2t\cdot ab\,.
\edm
For $t=2/3$, G.\ Jensen proved that this metric is Einstein, and Th.\
Friedrich showed that it carries a homogeneous spin structure and that
it admits two Riemannian Killing spinors
\cite{Friedrich80} and thus realizes the equality case in his
estimate for the first eigenvalue of the Dirac operator.
A detailed investigation of this family of metrics from the point of
view of metric connections with torsion may be found in \cite{Agricola03};
in particular, we refer to these two papers for all proofs of formulas
given below (however, we will write down whatever is needed to follow
our argument).
Denote by $E_{ij}$ the standard basis of $\so(4)$. Then the elements
 \bdm
 Z_1\, :=\, E_{13},\ Z_2 \,:=\, E_{14},\ Z_3\,=\,E_{23},\ Z_4\,=\,
 E_{24},\ Z_5\,=\,\frac{1}{\sqrt{2s}}\,E_{12}
 \edm
form an orthonormal basis of $\m$.
Identifying $\m$ with $\R^5$ via the chosen basis, the isotropy
representation of an element $g(\theta)=\left[\ba{cc}\cos\theta& -\sin\theta\\
\sin\theta&\cos\theta \ea\right]\in H=\SO(2)$
and its lift to the $4$-dimensional spinor
representation $\kappa: \Spin(\R^5)\ra \GL(\Delta_5)$ can be computed,
 \bdm
 \Ad g(\theta)\ =\  \left[\ba{ccccc}
 \cos\theta & -\sin\theta & 0 & 0 & 0 \\
 \sin\theta & \cos\theta & 0 & 0 & 0 \\
 0 & 0 & \cos\theta & -\sin\theta & 0 \\
 0 & 0 & \sin\theta & \cos\theta & 0 \\
 0 & 0 & 0 & 0 & 1
 \ea\right],\quad
 \kappa\big(\Adtilde g(\theta)\big)\ =\ \left[\ba{cccc}
 e^{i\theta} & 0 & 0 & 0 \\
 0 & e^{-i\theta} & 0 & 0 \\
 0 & 0& 1 & 0\\ 0 & 0 & 0 & 1 \ea\right]\,.
\edm
Thus, the basis elements $\psi_3$ and $\psi_4$ of $\Delta_5$ define sections of
the spinor bundle $S=G\x_{\kappa(\Adtilde)}\Delta_5$ if viewed as
constant maps $G\ra\Delta_5$. In fact, for $t=2/3$,
$\psi^{\pm}:= \pm  i\psi_3 +\psi_4 $
are exactly the Riemannian Killing spinors from \cite{Friedrich80}.
For the undeformed metric ($t=1/2$), these two spinors are parallel.
In \cite[Prop. 3]{Jensen75}, the author computed the abstract formulas for
the map $\Lambda_{\m}^{g}: \m\cong\R^5\ra\so(5)$
 defining the Levi-Civita connection
in the sense of Wang's Theorem \cite[Ch. X, Thm 2.1]{Kobayashi&N2}.
In our example, this yields (viewed as endomorphisms of $\R^5$):
 \bdm
 \Lambda_{\m}^{g}(Z_1)\,=\, \sqrt{\frac{t}{2}}E_{35},\quad
 \Lambda_{\m}^{g}(Z_2)\,=\, \sqrt{\frac{t}{2}}E_{45},\quad
 \Lambda_{\m}^{g}(Z_3)\,=\, - \sqrt{\frac{t}{2}}E_{15},\quad
 \Lambda_{\m}^{g}(Z_4)\,=\, - \sqrt{\frac{t}{2}}E_{25},
\edm\bdm
 \Lambda_{\m}^{g}(Z_5)\,=\, \frac{1-t}{\sqrt{2t}}(E_{13}+E_{24})\,.
 \edm
The space $\m$ has a preferred direction, namely
$\xi= Z_5$, which is fixed under the isotropy representation.
Denote its dual $1$-form, $\eta(X)=\lan Z_5, X\ran$ by $\eta$.
As discussed in \cite{Agricola03}, there exist three
almost contact metric structures intertwining the
isotropy representation, and their  characteristic connections
with torsion coincide (see \cite{Friedrich&I1} for general results).
To fix the ideas, choose for example the skew-symmetric endomorphism
$\varphi: TM\ra TM$ defined by $\vphi(Z_1)=-Z_3, \ \vphi(Z_2)=-Z_4,
\vphi(Z_5)=0$. The differential of its fundamental form
$F(X,Y):= \lan X,\varphi(Y) \ran$ and  its Nijenhuis tensor vanish.
Thus, the Stiefel manifold $V_{4,2}$ admits
a characteristic connection $\nabla^c$ with torsion
 \bdm
 T\,=\, \eta\hut d\eta \, =\,- \sqrt{2t}\,(Z_1\hut Z_3+Z_2\hut Z_4)\hut Z_5\,.
 \edm
One checks that $T$ is parallel, hence the methods described here apply.
Furthermore, one has the following geometric data,
\bdm
\|T\|^2\ =\ 4t, \quad \mu\in\{0,\pm 2\sqrt{2t}\},\quad
\Scalg\ =\ 8-2t, \quad
 \Ric^g \ =\ \diag(2-t,2-t,2-t,2-t,2t).
\edm
Using the formulas for the Levi-Civita connection, one
checks that $\psi^\pm$ is an $\D$-eigenspinor to the eigenvalue
$\pm 1/\sqrt{2t}$. Since $\psi^\pm$ lies in the bundle $\Sigma_\mu$
with $\mu=\mp 2\sqrt{2t}$,
the universal eigenvalue estimate and the
twistorial eigenvalue estimate take the
numerical values
\bdm
\beta_{\univ}\ =\ 2(1-t),\quad \beta_{\tw}\ =\ \frac{5}{2}-\frac{25}{8}t .
\edm
Thus, the two estimates coincide for $t=4/9$; below this value,
the twistorial estimate is better, while above, the universal estimate
is to be preferred. For large $t$ -- corresponding to highly negative
curvature --  both estimates are not applicable. In the figure below,
these estimates and the known eigenvalue $\lambda(\D^2)=1/2t$ of $\D^2$
are drawn.
\bdm \psfrag{t}{$t$}
\psfrag{EW}{$\lambda(\D^2)$}
\psfrag{buniv}{$\beta_{\univ}$}\psfrag{btw}{$\beta_{\tw}$}
\psfrag{4/9}{$\frac{4}{9}$}
\includegraphics[width=7.5cm]{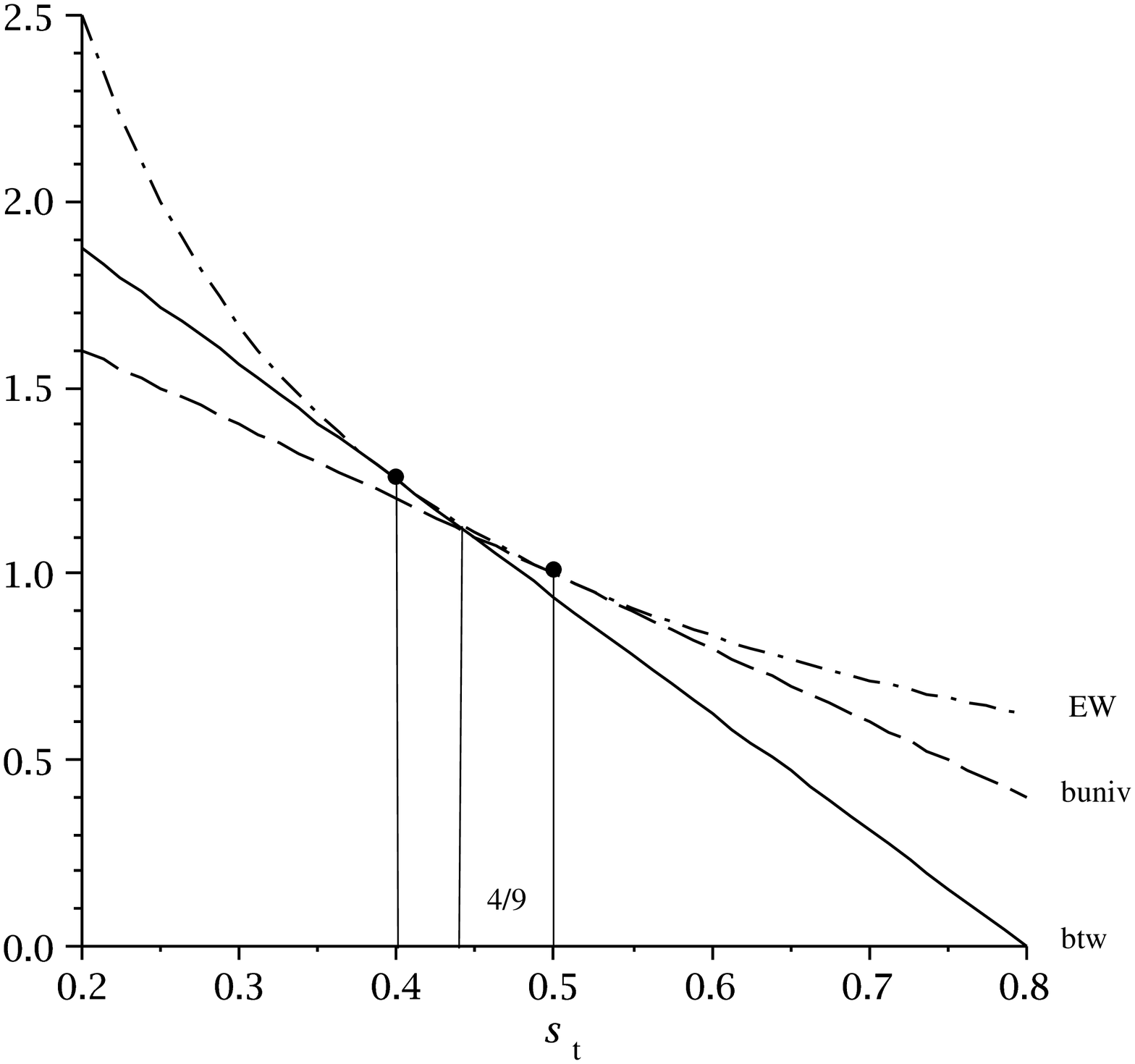}
\edm
 A priori,  $\lambda(\D^2)$ has no reason to be the smallest
eigenvalue. However, in the case $t=1/2$ the spinor fields 
$\psi^\pm$ are $\nabla^c$-parallel, and thus they realize the equality case in 
the universal estimate. For the twistorial estimate,
 one sees that $\beta_{\tw}$ becomes optimal for $t=2/5$, and
thus $\psi^\pm$ are automatically twistor spinors with torsion. A more detailed
computer computation reveals that more is true in this case: 
$\psi^\pm$ is a Killing spinor
with torsion to the Killing number\footnote{The example thus shows that 
the Killing number of a Killing spinor with torsion on a compact 
manifold can be of either sign, in contrast to the Riemannian case.} 
$\kappa =\pm \sqrt{5}/10$.
These are two of the four  solutions of the two quadratic
equations for $\kappa$ from Lemma \ref{gen-Killing-eq} (one equation
for each value of $\mu$), the other two being $\pm 3\sqrt{5}/10$
(for $\mu=\mp 2\sqrt{2t}$). These would yield a larger
$\D$-eigenvalue, thus they cannot correspond to twistor spinors.
One checks that the  Killing vector fields associated to $\psi^\pm$
by equation (\ref{Killing-VF}) are non-vanishing multiples of $\xi=Z_5$.

The example also illustrates the eigenvalue estimate
in single eigensubbundles $\Sigma_\mu$, as described in
Corollary \ref{cor.twistorial-est-mu}. The torsion $T$ has the eigenvalues
$0,\pm 2\sqrt{2t}$, thus  $\max(\mu^2_i)=8t$, and as we saw, the twistor
spinors with torsion lie in the  eigensubbundles corresponding the
the $\mu$ values for which the maximum is attained. On the other side,
Corollary \ref{cor.twistorial-est-mu} allows an eigenvalue estimate
on the remaining bundle with $\mu=0$. Since $n=5$, the twistorial eigenvalue
estimate takes the form
\bdm
\lambda\big(\D^2\big|_{\Sigma_0}\big) \ \geq\ \beta_\tw(0)\ =\
 \frac{5}{16}\Scalg\ =\  \frac{5(4-t)}{8}.
\edm
This is exactly the Riemannian estimate. It is optimal if and only if
there exists a Riemannian Killing spinor in $\Sigma_0$, which is never
the case. We can again compare this estimate with the universal estimate
from Theorem \ref{thm.universal-est},
\bdm
\lambda\big(\D^2\big|_{\Sigma_0}\big) \ \geq\ \beta_\univ(0)\ =\
\frac{1}{4}\Scalg + \frac{1}{8}\|T \|^2 \ =\ 2.
\edm
Hence, the twistorial estimate lies above the universal estimate
for $t\leq 4/5$. Since equality in the universal estimate is obtained
for $\nabla^c$-parallel spinors, it can presumably  not be obtained as well, though
we have no strict argument that excludes the existence of parallel spinors in
$\Sigma_0$.
\end{exa}
\begin{exa}[A $7$-dimensional manifold with Killing spinors with torsion]
\label{exa.Stiefel-2}
We shortly describe a second example, relatively similar to the previous one,
hence we will not give so many details. Consider the $7$-dimensional
Stiefel manifold $V_{5,2}=\SO(5)/\SO(3)$, with $\SO(3)$ embedded as
upper diagonal $(3\x 3)$-block. 
 The complement $\m$ of $\so(3)$ inside
$\so(5)$ splits under the isotropy representation into two copies of
the defining representation and a one-dimensional
trivial summand, $\m=\m_3 \oplus \m_3\oplus \m_1$. We define a new
metric on $V_{5,2}$ by deforming the
Killing form in direction of $\m_1$ by a factor $t>0$,
$g_t := \beta \big|_{\m_3\oplus\m_3} + t\beta\big|_{\m_1}$.
This manifold is known to be Einstein Sasaki and to have two Riemannian 
Killing spinors for $t=3/2$, see \cite{FKMS}, \cite{Kath00}.
$V_{5,2}$ carries an almost metric contact structure in direction  $\m_1=\R\cdot Z_7$
with vanishing Nijenhuis tensor and $\vphi=E_{14}+E_{25}+E_{36}$. It admits a 
characteristic connection $\nabla^c$ with torsion
\bdm 
T\,=\, \eta\hut d\eta \, =\,- \sqrt{t}\,(Z_1\hut Z_4+Z_2\hut Z_5+ Z_3\hut Z_6 )\hut Z_7\,.
\edm
One checks that $\nabla^c T=0$ for all $t$ and that the $T$ eigenvalues are
$3\sqrt{t}$ (multiplicity $2$) and $-\sqrt{t}$ (multiplicity $6$).  The lift of the isotropy representation
to the spin representation has two invariant spinors $\vphi_{\pm}\in 
\Sigma_{3\sqrt{t}}$, which will thus define global spinor fields on $V_{5,2}$.
After computing the formulas for the Levi-Civita connection, one
checks that $\vphi^\pm$ are $\D$-eigenspinors to the eigenvalue
$\pm -3/2\sqrt{t}$. Furthermore,
\bdm
\|T\|^2\ =\ 3t, \quad
\Scalg\ =\ 18-3t/2, \quad
 \Ric^g \ =\ \diag(3-t/2,\ldots, 3-t/2, 3t/2).
\edm
The universal eigenvalue estimate and the
twistorial eigenvalue estimate take the
numerical values
\bdm
\beta_{\univ}\ =\ \frac{9}{4}(2- t),\quad 
\beta_{\tw}\ =\ \frac{21}{4}-\frac{49}{16}t .
\edm
Equality is reached for $t=12/13$, hence the twistorial estimate
is better for metrics with $t<12/13$. Indeed, $\vphi_{\pm}$ are
Killing spinors with torsion for $t=42/49$ with the Killing number
$\kappa=-\sqrt{42}/56$.
\end{exa}
Further examples may be found in the second author's Ph.D.~thesis, see \cite{BB12}.
%
\section{The twistor equation in dimension $6$}\label{section.n6}
%
In dimension $6$, the twistor equation can be further reduced
to a Killing equation, thus leading to considerable simplifications.
For convenience, recall that the twistorial eigenvalue estimate
for $n=6$ amounts to
\bdm
\lambda\ \geq\ \frac{3}{10}\Scalg+\frac{1}{12}\|T\|^2 -\frac{1}{3}\mu^2.
\edm
\begin{lem}\label{lem.kommutator}
Assume $\nabla^c T=0$ und let $\psi$ be a twistor spinor for
$s=\frac{n-1}{4(n-3)}$. Then
 $\D$ and $T$ satisfy the relation
\bdm
\left[\D T+(1-\frac{6}{n})T\D\right]\psi
\ = \ \left[\frac{5-n}{n-3}T^2-\frac{2}{n-3}||T||^2\right]\psi.
\edm
\end{lem}
\begin{proof}
We start with the anticommutator relation from \cite{Friedrich&I1}, cited in
Theorem \ref{thm.Fr&I}, (2):
\bdm
D^sT+TD^s=dT+\delta T-8s\sigma_T-2\mathcal{D}^s.
\edm
For $\nabla^c$-parallel torsion and $s=1/4$, 
the three first terms on the right hand side vanish, hence
\bdm
D^cT+TD^c=-2\mathcal{D}^c\,.
\edm
Since $D^c=\D+\frac{1}{2}T$, this may be restated as
\bdm
\D T+T\D +T^2=-2\mathcal{D}^c\,.
\edm
The action of  $\mathcal{D}^c$ on a twistor spinor $\psi$
may be computed from the twistor equation
(Lemma \ref{lem.twistorglg}) with
 $s=\frac{n-1}{4(n-3)}$ and $X=e_i$, multiplying by $e_i\haken T$ and 
then summing over $i$,
\bdm
\mathcal{D}^c\psi+\frac{3}{n}T\D\psi+
\frac{1}{n-3}T^2-\frac{1}{n-3}||T||^2\ =\ 0.
\edm
Now one  obtains the desired result by inserting the expression for
$\mathcal{D}^c\psi$ in the previous relation.
\end{proof}
This relation has particularly interesting consequences for $n=6$.
\begin{cor}\label{cor.TSequalsKS}
Let $n=6$ and $\nabla^c T=0$. If $\psi$ is a twistor spinor for
$s=\frac{n-1}{4(n-3)} = \frac{5}{12}$
in the $T$ eigenbundle $\Sigma_\mu$, exactly one of the two
following cases holds:
\begin{enumerate}
\item $\mu=0$: either $T=0$ or $\psi=0$;
\item $\mu\neq 0$:  $\psi$ is a $\D$ eigenspinor
with eigenvalue
\bdm
\D\psi\ =\ - \frac{1}{3}\left[ \mu + 2\frac{\|T\|^2}{\mu}\right]\psi
\edm
and the twistor equation for $\psi$ and $s=5/12$ is equivalent to the
Killing equation
$\nabla^s\psi=\kappa X\cdot \psi$ for the same value of $s$
and  Killing number  $\displaystyle
\kappa= \frac{1}{9}\left[\frac{\|T\|^2}{\mu}-\mu\right]$, thus leading
to the Killing equation in its final form
\bdm
\nabla^c_X\psi - \frac{1}{18} \left[ \mu + 2\frac{\|T\|^2}{\mu}\right]X\cdot
\psi + \frac{1}{6} (X\wedge T)\psi \ =\ 0.
\edm
In particular, the scalar curvature is constant and satisfies the
relation
\bdm
\frac{3}{10}\Scal^g\ =\ \frac{4}{9}\mu^2 +\frac{13}{36}\|T\|^2
+\frac{4}{9}\frac{\|T\|^4}{\mu^2}\ .
\edm
\end{enumerate}
\end{cor}
\begin{proof}
 Lemma \ref{lem.kommutator} yields for $n=6$ and a twistor spinor
$\psi$
\bdm
\D T\psi \ =\ - \frac{1}{3}\left[T^2 + 2 \|T\|^2 \right]\psi.
\edm
By assumption,  $T\psi=\mu\psi$, hence
\bdm
\mu\D\psi\ =\ - \frac{1}{3}\left[\mu^2 +2 \|T\|^2 \right]\psi.
\edm
In case $\mu=0$, we get $\|T\|^2\psi=0$, hence the first claim.
For $\mu\neq 0$, we may divide by $\mu$ and
 $\psi$ is an eigenspinor of $\D$. We
insert this eigenvalue relation into the twistor equation from
Lemma \ref{lem.twistorglg},
\bdm
\nabla^c_X\psi - \frac{1}{18} \left[ \mu + 2\frac{\|T\|^2}{\mu}\right]X\cdot
\psi + \frac{1}{6} (X\wedge T)\psi \ =\ 0.
\edm
On the other hand, suppose that
 $\vphi$ is a Killing spinor in $\Sigma_\mu$ for $s=\frac{5}{12}$.
The Killing equation $\nabla^s_X\vphi=\kappa X\cdot\vphi$ implies that
 $\vphi$ is an eigenspinor of $D^s$ with eigenvalue $-n\kappa$.
For us, this means
\bdm
D^{\frac{5}{12}}\vphi\ =\ (\D+T)\vphi\ =\ -6\kappa\vphi .
\edm
Since $\vphi$ is also an eigenspinor for $\D$, this last
equation yields for the Killing number  $\kappa$ the value
\bdm
\kappa\ =\ \frac{1}{9}\left[\frac{\|T\|^2}{\mu}-\mu\right] .
\edm
With this value of  $\kappa$, the Killing equation for  $s=\frac{5}{12}$
becomes equivalent to
\bdm
\nabla^c_X\vphi - \frac{1}{18} \left[ \mu + 2\frac{\|T\|^2}{\mu}\right]X\cdot
\vphi + \frac{1}{6} (X\wedge T)\vphi \ =\ 0 .
\edm
Hence, every twistor spinor in  $\Sigma_\mu$  is necessarily a Killing
spinor.

The claim on the scalar curvature follows essentially from
Lemma \ref{lem.twistorglg}. For a twistor spinor $\psi$ in the
$T$ eigenbundle $\Sigma_\mu$, it yields
\bdm
\D^2\psi=\frac{3}{10}\Scal^g\cdot\psi+\frac{1}{12}\|T\|^2\psi-\frac{1}{3}\mu^2\psi .
\edm
On the other side, the $\D$ eigenvalue equation squared amounts to
\bdm
\D^2\psi=\frac{1}{9}(\mu^2+4\|T\|^2+\frac{4}{\mu^2}\|T\|^4)\psi .
\edm
A direct comparison leads to the relation for $\Scal^g$.
\end{proof}
\begin{NB}
The value given for the scalar curvature in the previous Corollary allows to
solve explicitly the general quadratic equation for the Killing number
$\kappa$ stated in Lemma \ref{gen-Killing-eq}. One obtains the possible
solutions
\bdm
\kappa_1\ =\ \frac{1}{9}\left[\frac{\|T\|^2}{\mu}-\mu\right] \
\text{ or } \
\kappa_2\ =\ - \frac{1}{9}\left[\frac{\|T\|^2}{\mu}+ 2\mu\right]
\edm
The previous Lemma thus shows that $\kappa_2$ cannot occur.
\end{NB}
\begin{exa}
Consider a $6$-dimensional nearly K\"ahler manifold
$(M^6,g,J)$ with its characteristic connection $\nabla^c$
(see also the discussion at the end of Section \ref{subsection.par-spin}).
These are
Einstein spaces of positive scalar curvature,
$\|T\|^2 = \frac{2}{15}\Scalg$, and $T$ has the eigenvalues
$\mu=0$ of multiplicity $6$ and $\mu=\pm 2\|T\|$, each with multiplicity $1$,
and the torsion is always parallel \cite{Alexandrov&F&S04}.
It is well-known that it has two Riemannian
Killing spinors $\vphi_{\pm}$ \cite{Friedrich&G85},
that these coincide with the $\nabla^c$-parallel 
spinors \cite[Thm.~10.8]{Friedrich&I1}, and that they lie in $\Sigma_{\pm 2\|T\|}$.
As observed before,  $\beta_\tw(\mu)= \beta_\univ
(\mu)=\frac{2}{15}\Scalg$: Thus, \emph{any} $\nabla^c$-parallel spinor has to be
a twistor spinor with torsion by Corollary \ref{cor.twistorial-est}.
Corollary \ref{cor.TSequalsKS} then implies that it is already a Killing
spinor with torsion. If $\mu=0$, any twistor spinor with torsion would
have to be a Riemannian twistor spinor, but it is kwown that these do not
exist in $\Sigma_0$. Hence, we proved:
\begin{thm}\label{thm.nK-all-spinors}
On a $6$-dimensional nearly K\"ahler manifold $(M^6,g,J)$ with its
characteristic connection $\nabla^c$, the following classes
of spinors coincide:
\begin{enumerate}
\item Riemannian Killing spinors,
\item $\nabla^c$-parallel spinors,
\item Killing spinors with torsion $(s=5/12)$,
\item Twistor spinors with torsion $(s=5/12)$.
\end{enumerate}
Furthermore, there is exactly one such spinor $\vphi_{\pm}$  in each of
the subbundles $\Sigma_{\pm 2\|T\|}$, their Killing numbers (with torsion)
are $\displaystyle\kappa = \mp \frac{\|T\|}{6}= \mp \frac{1}{3} \sqrt{\frac{\Scalg}{30}} $
and their $\D$ eigenvalues are 
$\mp \|T\|= \mp \displaystyle\sqrt{\frac{2 \,\Scalg}{15}} $.
\end{thm}
\end{exa}
\noindent
Further examples of $6$-dimensional manifolds with Killing spinors with
torsion well be discussed in a forthcoming paper.
%
\section{Twistorial estimates for manifolds with reducible holonomy}
%
%
\noindent
Recall that we assume that $(M^n,g)$ is an oriented Riemannian
manifold endowed with a metric connection $\nabla$ with skew-symmetric torsion
$T\in\Lambda^3(M^n)$. The holonomy group $\Hol(M^n;\nabla)$
(sometimes just abbreviated $\Hol(\nabla)$ if no confusions are possible)
is then a subgroup of $\SO(n)$, and we shall assume that it is a closed subgroup to
avoid pathological cases. In order to distinguish it from the torsion,
the tangent bundle and its subbundles will be denoted by $\TM^n$,
$\T_1,\T_2\ldots$.
\begin{dfn}[parallel distribution]
Let $x\in M^n$ and $\T_x$  be a $\Hol(M^n;\nabla)$-invariant
subspace of $\TM^n$. For any other point $y$, choose a curve $\gamma$ from
$x$ to $y$ and denote by $\T_y$ the image of $\T_x$ in $\T_y M^n$ under parallel
transport along $\gamma$. The subspace  $\T_y$ does not depend on the choice
of $\gamma$, for any other  curve $\tilde{\gamma}$ defines a closed loop
through $x$ by $\mu:=\tilde{\gamma}^{-1}\gamma$ and $\T_x$ is by assumption
invariant under parallel transport along $\mu$, meaning
$\tilde{\gamma}^{-1}\gamma(T_x)=T_x$. This implies
$\gamma(T_x)=\tilde{\gamma}(T_x)$ as stated. In particular, any
$\Hol(M^n;\nabla)$-invariant subspace $\T_x\subset\T_x M^n $ defines a
distribution $\T\subset \TM^n$. Any distribution occuring in this way will be
called \emph{parallel}.
\end{dfn}
\noindent
The proof of the following basic lemma carries over from Riemannian geometry
 without modifications (see for example \cite[Prop.\,5.1]{Kobayashi&N1}).
\begin{lem}\label{basic-lemma}
Let $\T\subset\TM^n$ be a parallel distribution and $Y\in \T$. For any $X\in
\TM^n$, $\nabla_X Y$ is again in $\T$; in particular,
$R(X_1,X_2)Y\in\T$ for any $X_1,X_2$.
\end{lem}
For a torsion free connection, this property implies of course that any
parallel distribution is involutive;  but for general metric connections this
conclusion does not hold anymore.

Let $\T$ be a parallel distribution, $\mathcal{N}$ its orthogonal distribution defined
by $\mathcal{N}_x:=\T_x^\perp$ in every point $x\in M^n$. The fact that all elements
of $\Hol(M^n;\nabla)$ are orthogonal transformations implies that
$\mathcal{N}$ is again a parallel distribution. Thus, the tangent bundle
splits into an orthogonal sum of parallel distributions ($n_i:=\dim\T_i$)
\bdm
\TM^n\ =\ \T_1\oplus\ldots\oplus \T_k, \text{ and }
\Hol(M^n;\nabla)\subset \Orth(n_1)\x\ldots\x\Orth(n_k)\subset\SO(n).
\edm
We assume that every distribution $\T_i$ is again orientable and that
the holonomy preserves the orientation, i.\,e.~\emph{we assume}
\bdm
\Hol(M^n;\nabla)\subset \SO(n_1)\x\ldots\x\SO(n_k).
\edm
In this case, every
parallel distribution $\T_i$ defines a parallel
$n_i$-form: if $X_1,\ldots,X_{n_i}$ is a generating frame for $\T_i$, then
$\alpha_i:=X_1\wedge \ldots\wedge X_{n_i}$ is a differential form and spans a
$1$-dimensional $\SO(n_i)$-invariant
subspace of $\Lambda^{n_i}(M^n)$. This is  necessarily the trivial
representation, meaning that $\alpha_i$ is invariant under parallel transport.
We agree that any  orthonormal frame $e_1,\ldots,e_m$  of
$\TM^n$ shall respect the splitting in parallel distributions (i.\,e.~no
$e_k$ has parts in different distributions $\T_i$),  and that
$e^i_{1},\ldots,e^i_{n_i}$ is to denote an orthonormal frame of $\T_i$,
$i=1,\ldots, k$.

We will now describe the `block structure' of the curvature.
All curvatures are meant to be those of the connection $\nabla$.
Recall that the curvature tensor of a metric connection has  the symmetry
property
\bdm
g(\mathcal{R}(X,Y)W_1,W_2)\ =\ -g(\mathcal{R}(X,Y)W_2,W_1).
\edm
Since the distributions $\T_i,\T_j$ are orthogonal,
Lemma \ref{basic-lemma} implies for any vector fields $X,Y$ that
\be\label{sym1}
g(\mathcal{R}(X,Y)\T_i,\T_j)=0 \text{ if } i\neq j.
\ee
Furthermore, the Ambrose-Singer theorem implies that the curvature operator
$R(X,Y)$ vanishes if  $X\in\T_i,\ Y\in \T_j,\ i\neq j$,
\be\label{sym2}
R(\T_i,\T_j)\ =\ 0 \ \text{ for } i\neq j,
\ee
since the holonomy group is generated by all curvature operators
\cite[Thm 10.58]{Besse87}. We now consider  the Ricci curvature.
\begin{prop}\label{curv-splitting}
 The Ricci tensor has block structure,
\bdm
\Ric\ =\ \left[ \begin{array}{c|c|c} \Ric_1 & 0&\\ \hline 0 &\ddots& 0\\
    \hline  & 0 & \Ric_k
 \end{array}\right],
\edm
  i.\,e.~$\Ric(X,Y)\neq 0$ can only happen if  $X,Y\in\T_i$ for some $i$.
\item The scalar curvature splits into `partial scalar curvatures'
$\Scal_i:=\tr\, \Ric_i$, and ${\displaystyle\Scal=\sum_{i=1}^k \Scal_i}$.
\end{prop}
\begin{proof}
A summand of the Ricci tensor $\Ric(X,Y)=\sum_{m=1}^n R(e_m,X,Y,e_m)$
can only be non trivial if the vectors $(e_m, X)$ lie in the same $\T_i$
by equation  (\ref{sym2}), and if the vectors  $(e_m,Y)$ lie in the same $\T_j$
by equation (\ref{sym1}). But since the vector $e_m$ is the same in both cases
(and thus cannot lie in two different distributions),
we conclude that $e_m, X,$ and $Y$ all have to lie in some $\T_i$.
 Thus we can define
\bdm
\Ric_i(X,Y)\ :=\ \sum_{m=1}^{n_i} R(e^i_{m},X,Y,e^i_{m}).
\edm
Then, $\Ric=\sum_{i=1}^k\Ric_i$ and it has the stated block structure. The
partial scalar curvatures are now just the traces of
these partial Ricci tensors.
\end{proof}
\begin{NB}
Observe that the partial Ricci tensor vanishes in directions of parallel
vector fields ($n_i=1$).
\end{NB}
\begin{NB}
Be cautious that despite of the block structure of the Ricci curvature,
one has in general that  $R(X,Y,U,V)\neq 0$ if $X,Y\in \T_i,\ U,V\in  \T_j$
for $i\neq j$.
\end{NB}
Let us now consider, as in the first part of this paper, the $1$-parameter
family of connections
\bdm
\nabla^s_X Y \ =\ \nabla^g_X Y + 2s\, T(X,Y,-).
\edm
All quantities (curvature etc.)
belonging to the connection $\nabla^s$ will carry an upper index $s$.
We make the following crucial assumptions:
\begin{dfn}\label{dfn.geom-redpargeom}
Assume that
\begin{enumerate}
\item  there exists a value $s_0$ such that
$\nabla^{s_0}T=0$; without loss of generality, we will assume that
the torsion is normalized in such a way that $s_0=1/4$.
Instead of $\nabla^{1/4}$, we will write $\nabla^c$,
\item  the tangent bundle $\T M^n=\bigoplus_{i=1}^k \T_i$ splits into 
$\nabla^{s}$-parallel distributions $\T_i$ for all
parameters $s$ and  $\Hol(M^n;\nabla^{s})\subset
\SO(n_1)\x\ldots\x\SO(n_k)$,
\item the torsion splits into a sum  $T=\sum_{i=1}^k T_i$, $T_i\in\Lambda^3(\T_i)$. 
\end{enumerate}
We observe that the conditions contain some redundancy:
if the torsion splits as described, then it is sufficient to assume
that the tangent bundle is a sum of parallel distributions \emph{for
one parameter $s$}.

By de Rham's Theorem, $(M^n,g)$ is then locally a product of
Riemannian manifolds, i.\,e.~the universal cover $\tilde{M}$ of $M$ splits
into $\tilde{M}=M_1 \x \ldots\x M_k$ with $\dim M_i=n_i$.
A result of Cleyton and Moroianu \cite{Cleyton&Moroianu}
implies that each $T_i$ satisfies $\nabla^cT_i=0$, and in fact, each $T_i$ 
is the projection to $M$ of a $3$-form in $\Lambda^3 (M_i)$. 
In the sequel, we shall
call a manifold satisfying these
assumptions a \emph{geometry with reducible parallel torsion}.
\end{dfn}
\begin{NB}\label{rem.EVs-T2}
Since the forms $T_i$ live on disjoint distributions, $T_i T_j=-T_j T_i$
for $i\neq j$ in the Clifford algebra. This implies $T^2=\sum_{i=1}^k T_i^2$
and hence every eigenvalue $\mu^2$ of $T^2$ is a sum of eigenvalues of
the single $T_i^2$. By the orthogonality of the distributions $\T_i$,
the identity $\|T\|^2=\sum_{i=1}^k \|T_i\|^2$ holds.
\end{NB}
We can deduce some further curvature properties from this assumption:
\begin{lem}\label{curv-splitting-2}
If $(M,g)$ carries a geometry with reducible parallel torsion, 
$\mathcal{R}^s(X,Y,Z,V)$ can only be non-zero if all vectors lie
in the same subspace $\T_i$ for some $i$.
Furthermore, the $4$-form
$\sigma_T$ splits in $\displaystyle \sigma_T=\sum_{i=1}^k \sigma_i$,
where $\sigma_i:=\sigma_{T^i}$.
\end{lem}
\begin{proof}
The splitting $T=\sum T_i$
implies $\sigma_T=\sum \sigma_i$, $\sigma_i:=\sigma_{T^i}$, by definition
of $\sigma_T$. 
From Theorem \ref{thm.curv-splitting}, we know that the Bianchi identity in this
case reads 
\bdm
\stackrel{X,Y,Z}{\mathfrak{S}} \mathcal{R}^s(X,Y,Z,V)\ =\
s \left[6-8s\right]\, \sigma_T(X,Y,Z,V).
\edm
Consider now vectors $X,Y,U,V$. In order for $\mathcal{R}^s(X,Y,U,V)$ to be possibly
non zero, we need to assume that $X,Y\in \T_i,\ U,V\in T_j$ for some
indices $i,j$. The elements $\mathcal{R}^s(Y,U,X,V)$ and 
$\mathcal{R}^s(U,X,Y,V)$, however,
vanish from equations  (\ref{sym2}) and (\ref{sym1}), thus
we are left in this case with
\bdm
\mathcal{R}^s(X,Y,U,V) \ =\ s(6-8s)\, \sigma_T(X,Y,Z,V).
\edm
But  $\sigma_T(X,Y,Z,V)$ can only be different from zero if
all vectors lie in the same $\T_i$.
\end{proof}
 Assuming that $M^n$ is spin,
the curvature $R_\Sigma$ of the spinor bundle $\Sigma M^n$ for the connection
$\nabla^s$
\bdm
R^s_\Sigma (X,Y)\psi\ :=\ \nabla^s_X\nabla^s_Y\psi - \nabla^s_Y\nabla^s_X\psi
-\nabla^s_{[X,Y]}\psi
\edm
is related to the curvature $\mathcal{R}^s$ of the tangent bundle $\TM^n$ by
\bdm
R^s_\Sigma (X,Y)\psi\ =\ \frac{1}{2}R^s(X\wedge Y)\cdot\psi,
\edm
where we interpret the curvature transformation $\mathcal{R}^s$ as an endomorphism
on $2$-forms: $\mathcal{R}^s(e_i\wedge e_j):=\sum_{k<l}R^s_{ijkl}e_k\wedge e_l$. This allows
to draw conclusions on $R^s_\Sigma$ from the described splittings. 
%

\subsection*{Partial Schr\"odinger-Lichnerowicz formulas}
\noindent
%
The splitting of the tangent bundle makes a certain amount of bookkeeping
unavoidable. Let $p_i$ denote the orthogonal projection from $\TM^n$ onto
$\T_i$ and define the `partial connections'
\bdm
\nabla^{s,i}_X\ :=\ \nabla^s_{p_i(X)} ,\quad \text{hence }
\nabla^s\ =\ \sum_{i=1}^k \nabla^{s,i}.
\edm
They induce the notions of `partial Dirac operators' and `partial spinor Laplacians'
($\mu$ is the usual Clifford multiplication) through
\bdm
D_i^s\ :=\ \mu\circ\nabla^{s,i}, \quad D\ =\ \sum_{i=1}^k D^s_i,\quad
\Delta^s_ i \ :=\ (\nabla^{s,i})^*\nabla^{s,i},\quad \Delta^s \ =\
\sum_{i=1}^k \Delta^s_i.
\edm
At a fixed point $p\in M^n$ we choose orthonormal bases
$e^i_1,\ldots,e^i_{n_i}$  of the distributions $\T_i$ ($i=1,\ldots,k$) such that
$(\nabla^s_{e^i_m}e^j_l)_p = 0 $ for all suitable indices $i,j,m,l$.
Note that our chosen basis has the
properties $[e^i_m,e^j_l]=-T(e^i_m,e^j_l)$ and $\nabla^g_{e^i_m}e^i_m=0$. It is
convenient (and consistent with the notation introduced above) to abbreviate
$\nabla^s_{e^i_m}$ by $\nabla^{s,i}_m$. The partial Dirac and Laplace operators
may then be expressed as
\bdm
D^s_i\ :=\ \sum_{m=1}^{n_i} e^i_{m}\nabla^{s,i}_m,\quad
\Delta^s_i \ :=\ -\sum_{m=1}^{n_i} \nabla^{s,i}_m\nabla^{s,i}_m.
\edm
The divergence term of the Laplacian vanishes because of
$\nabla^g_{e^i_m}e^i_m=0$.
We compute the squares of the partial Dirac operators $D_i$ and their
anticommutators. To formulate the statement, set
\bdm
\mathcal{D}^s_i\ :=\ \sum_{m=1}^{n_i} (e^i_m\haken T_i)\cdot \nabla^{s,i}_{e^i_m}\psi
\edm
in full analogy to the classical case without splitting of $\T M$ ($k=1$).
\begin{prop}\label{partial-SL-1}
%
Assume that $M$ carries a geometry with reducible parallel torsion.
The partial Dirac operators $D^s_i$ then satisfy the identities
\begin{enumerate}
\renewcommand{\labelenumi}{(\roman{enumi})}
\item ${\displaystyle (D^s_i)^2 \, =\, \Delta^s_i + s (6-8s)\, {\sigma}_i
 -4s \mathcal{D}^s_i  + \frac{1}{4}\Scal^s_i} $,
\item $\displaystyle D^s_iD^s_j+D^s_jD^s_i = 0$ for $i\neq j$,
\item $\displaystyle (D^{s/3}_i)^2\ =\ 
\Delta^s_i+ 2s\,\sigma_i  +\frac{1}{4}\Scal^g_i-2s^2\|T_i\|^2$.
\end{enumerate}

\end{prop}
\begin{proof}
%
For the first identity, let $k$ and $l$ be indices running
between $1$ and $\dim \T_i=n_i$. We  split the sum into
terms with $k=l$ and $k\neq l$,
\bea[*]
(D^s_i)^2\psi & =& \sum_{k,l=1}^{n_i} e^i_k\nabla^{s,i}_k e^i_l\nabla^{s,i}_l\psi
\ =\  -\sum_{k=1}^{n_i} \nabla^{s,i}_k\nabla^{s,i}_k\psi+\sum_{k\neq l}
e^i_ke^i_l\nabla^{s,i}_k\nabla^i_l\psi\\
& =& \Delta^s_i+\sum_{k< l}e^i_ke^i_l(\nabla^{s,i}_k\nabla^{s,i}_l-
\nabla^{s,i}_l\nabla^{s,i}_k)\psi
\eea[*]
und express the second term through the curvature in the spinor bundle,
\bdm
(D^s_i)^2\psi \, =\, \Delta^s_i\psi + \sum_{k<l} e^i_ke^i_l\left[
\mathcal{R}^s_\Sigma(e^i_k,e^i_l) - \nabla^{s,i}_{T(e^i_k,e^i_l)}\right] \psi.
\edm
$\mathcal{R}^s_\Sigma$ in turn can be expressed through the curvature $R$, and by
Lemma \ref{curv-splitting-2}, only terms with all four vectors inside
$\T_i$ can occur:
\bdm \sum_{k<l} e^i_ke^i_l \mathcal{R}^s_\Sigma(e^i_k,e^i_l)\ =\
\frac{1}{2}\sum_{k<l} e^i_ke^i_l \mathcal{R}^s(e^i_k\wedge e^i_l)\cdot \psi\
=\ \frac{1}{2}\sum_{k<l, p<q} \mathcal{R}^s(e^i_k,e^i_l,e^i_p,e^i_q)
e^i_ke^i_le^i_pe^i_q\psi.
\edm
The summands with same index pairs add up to  half the partial
scalar curvature,
while totally different indices  yield the Clifford multiplication by 
the $4$-form $\sigma^i$.
Index pairs with one common index add up to zero, because
the Ricci tensor is symmetric (Theorem \ref{thm.curv-splitting}). 
The third identity follows by a routine calculation.

For the second identity and $i\neq j$, we proceed similarly.
However, there is no  diagonal term resulting in an analogue of
the Laplacian,  hence
\bdm
D^s_i D^s_j + D^s_j D^s_i \ =\ 
\sum_{k,l}e^i_k e^j_l \left[\mathcal{R}^s_\Sigma(e^i_k,e^j_l) - 
\nabla^{s}_{T(e^i_k,e^j_l)}\right].
\edm
But the mixed curvature operator vanishes as observed in equation
(\ref{sym2}), and $T(\T_i,\T_j)$ is zero as well by the assumption that $T$
does not contain any mixed terms.
\end{proof}
%
\noindent
The second identity has a crucial consequence:  all the operators
$(D^s)^2, (D^s_1)^2,\ldots, (D^s_k)^2$ can be simultaneously diagonalized.
\begin{lem}\label{sum-Diracs}
For all parameters $s$, $\displaystyle (D^s)^2=\sum_{i=1}^k (D^s_i)^2$ and
$(D^s_i)^2\, (D^s_j)^2 = (D^s_j)^2\, (D^s_i)^2\ \forall\, i\neq j$.
\end{lem}
\noindent
In particular, any eigenvalue $\lambda$ of $(D^s)^2$ is the sum of 
eigenvalues $\lambda_i$ of $(D^s_i)^2$, 
$\displaystyle \lambda =\sum_{i=1}^k \lambda_i $.
%
\subsection*{Adapted Twistor Operator}\label{sec:twis}\noindent
%
%
We define a twistor operator $P^s:\ \Gamma(\Sigma M)\ra \Gamma (\T M ^{\ast}\otimes
\Sigma M)$ adapted to the splitting $\T M^n=\bigoplus_{i=1}^k \T_i$ 
of the tangent bundle by
\be\label{eq.adapted-twistor}
P^s\psi \ = \ \nabla^s \psi +
\sum_{i=1}^{k}\frac{1}{n_i}\sum_{l=1}^{n_i}e^i_l \otimes
e^i_l\cdot D^s_i\psi. 
\ee
By  a simple computation, one checks that
\bqr 
\|P^s\psi\|^2 \  = \ \langle(\Delta^s - \sum_{i=1}^k
                   \frac{1}{n_i}(D^s_i)^2)\psi, \psi\rangle.\label{twistor}\eqr
\begin{thm}[Twistorial eigenvalue estimate for products]%
\label{thm.twistorial-est-prod}
Assume that $M$ carries a geometry with reducible parallel torsion, and that
the dimensions of the subbundles $\T_i$ are ordered by ascending
dimensions, $n_1\leq n_2\leq \ldots\leq n_k$.
The smallest eigenvalue $\lambda$ of $\D^2$
satisfies the inequality
\bdm\tag{$*$}
\lambda \ \geq\ \frac{n_k}{4(n_k-1)}\Scal^g_{\min}+ \frac{n_k(n_k-5)}{8(n_k-3)^2}
\|T\|^2
+\frac{n_k(4-n_k)}{4(n_k-3)^2}\max(\mu_1^2,\ldots,\mu_k^2).
\edm
Let $\tilde{M}=M_1\x\ldots \x M_k$ be the universal cover of $M$, $\dim M_i=n_i$,
and $\tilde{s}:=\frac{n_k-1}{4(n_k-3)}$.
Equality holds in $(*)$ if and only if the following conditions are
satisfied:
\begin{enumerate}
\item The Riemannian scalar curvature of $(M,g)$ is constant,
\item There exists a twistor spinor with torsion for $\tilde{s}$
on $M_k$,
\item For $i=1,\ldots, k-1$:
\begin{enumerate}
\item If $n_i<n_k$: there exists  a $\nabla^{\tilde{s}}$-parallel
spinor on $M_i$, 
\item  If $n_i=n_k$: there exists  a $\nabla^{\tilde{s}}$-parallel or 
a twistor spinor with torsion for $\tilde{s}$ on $M_i$, 
\end{enumerate}
\item these spinors lie in the subbundle
$\Sigma_\mu(M_i)$ corresponding to the largest eigenvalue of $T^2_i$ ($i=1\ldots,k$).
\end{enumerate}
\end{thm}
\begin{proof}
From Theorem \ref{thm.gSLF}, we know that
\bdm
(D^{s/3})^2\ =\ \Delta^s - s T^2+\frac{1}{4}\Scal^g+
(s-2s^2)\|T\|^2.
\edm
By integrating, we conclude together with equation (\ref{twistor})
 \bqrs \int\langle(D^{s/3})^2\psi, \psi\rangle\rangle
  &=&\int\langle\Delta^s \psi, \psi\rangle
      -s\int\langle T^2 \psi, \psi\rangle
      +\frac{1}{4}\int\langle \Scalg \psi, \psi\rangle
      +(s-2s^2)\int\|T\|^2|
      \|\psi\|^2 \\
  &=&\int\|P^s\psi\|^2 +
  \sum_{i=1}^k\int\frac{1}{n_i}\|D^s_i\psi\|^2\\
   &&\hspace*{1cm}   -s\int\langle T^2 \psi, \psi\rangle
      +\frac{1}{4}\int\langle \Scalg \psi, \psi\rangle
      +(s-2s^2)\int\|T\|^2|
      \|\psi\|^2.
\eqrs
Adding $-\int\langle\frac{1}{n_k}(D^s)^2\psi,\psi\rangle$ on 
both sides
\bea[*]
\int\langle\left[(D^{s/3})^2 -
\frac{1}{n_k}(D^s)^2\right] \psi, \psi\rangle
  &=&\int\|P^s\psi\|^2
       +\int\langle\left[\sum_{i=1}^k\frac{1}{n_i}(D^s_i)^2 -
       \frac{1}{n_k}(D^s)^2\right]\psi,\psi\rangle\\
   &&    -s\int\langle T^s \psi, \psi\rangle
      +\frac{1}{4}\int\langle \Scalg \psi, \psi\rangle
      +(s-2s^2)\int\|T\|^2|
      \|\psi\|^2.
\eea[*]
Using the same techniques as in Lemma \ref{lem.fund-calc} and
Theorem \ref{twistorial-int-FL} we have
\[(D^{s/3})^2 -
\frac{1}{n_k}(D^s)^2 = \frac{n_k - 1}{n_k}(\D)^2 +
\frac{4s^2}{1-n_k}T^2\quad \text{with}\,\,s = \frac{n_k -1}{4(n_k
-3)}.\]
We also note from Lemma \ref{sum-Diracs} that
$\displaystyle(D^s)^2 = \sum_{i=1}^k (D^s_i)^2$; hence,
if we set  $\tilde{s} := \frac{n_k -1}{4(n_k -3)}$, the above formula implies
\begin{gather*}
\frac{n_k-1}{n_k}\int\langle(\D)^2\psi, \psi\rangle =
\int\|P^s\psi\|^2
       +\int\langle\left[\sum_{i=1}^{k-1}
       \left(\frac{1}{n_i}-\frac{1}{n_k}\right)(D^s_i)^2\right]\psi,\psi\rangle\\
 + \frac{(1-n_k)(n_k -4)}{4(n_k -3)^2}\int\langle T^2 \psi, \psi\rangle
      +\frac{1}{4}\int\langle \Scalg \psi, \psi\rangle
      +\frac{(n_k -1)(n_k - 5)}{8(n_k - 3)^2}\int\|T\|^2|
      \|\psi\|^2.
\end{gather*}
Since the first two terms on the right hand side of the first line are
clearly $\geq 0$, the estimate follows.
The equality case is obtained if the scalar curvature is constant,
the eigenvalue $\mu^2$ of $T^2$ is maximal and
\bdm
P^{\tilde s}\psi\ =\ 0,\quad \text{and}\  (D^{\tilde s}_i)^2 \psi\ =\ 0 \text{ if }
n_i< n_k. 
\edm
By the splitting of the tangent bundle, $P^{\tilde s}\psi=0$ is equivalent
to
\bdm
\nabla^{\tilde{s}, i} \psi +\frac{1}{n_i}\sum_{l=1}^{n_i}e^i_l \otimes
e^i_l\cdot D^{\tilde s}_i\psi \ = \ 0  \quad \forall i=1,\ldots, k. 
\edm
If $D^{\tilde s}_i\psi =0$ (for example, if $n_i< n_k$), this implies
$\nabla^{\tilde{s}, i} \psi=0$. If not, $n_i=n_k$ has to hold and
\bdm
\nabla^{\tilde{s}, i}_X \psi +\frac{1}{n_i}p_i(X)\cdot D^{\tilde s}_i\psi \ = \ 0,
\edm
where we recall that $p_i$ denoted the projection $\T M \ra \T_i$.
In order to obtain the corresponding spinor fields on the factors $M_i$,
one argues as in \cite{Alexandrov07}: Let $f_i:\ M_i\ra \tilde{M}$ be the inclusion,
$q:\ \tilde{M}\ra M$ the projection. The pullback $(q\circ f_i)^* \Sigma M$ is a
Clifford bundle over $M_i$, hence it is a sum of finitely many copies of $\Sigma M_i$, and the connection $\nabla^s$ with torsion
$4sT$  pulls back to the connection with torsion $4sT_i$. The pullbacks of
the spinor fields then satisfy the corresponding field equations.
For the last claim, observe that $T^2=\sum_{i=1}^k T_i^2$ as observed in
Remark \ref{rem.EVs-T2}, hence the maximum is reached if and only if
the eivenvalue of each single $T_i^2$ is maximal.
\end{proof}
\noindent
\begin{NB}
The coefficient of the scalar curvature is strictly  decreasing
with growing $n$, while the other two coefficients (of $\|T\|^2$ and
$\max(\mu_1^2,\ldots, \mu^2_k)$) 
are first increasing (for $n\geq 5$ resp.~$n\geq 6$), then decreasing
for larger $n$ ($n\gtrsim 15$)  and reach quickly their 
limits $1/8$ and $-1/4$.  Thus, this estimate becomes better for products 
where the scalar curvature dominates the other two terms. We believe
that the estimate will be particularly useful if $M$ is locally  a product
of manifolds of the same dimension (though not necessarily carrying the same
geometry), because it has a `nicer' equality case. 
\end{NB}
\begin{NB}
Theorem \ref{thm.twistorial-est-prod} generalizes the main result
by E.\,C.~Kim and B.~Alexandrov (\cite{ECKim04}, \cite{Alexandrov07})
for the Riemannian Dirac operator on locally reducible Riemannian manifolds.
They proved that the first eigenvalue $\lambda^g$ of the square of the
Riemannian Dirac operator $(D^g)^2$ satisfies
\be\label{eq.bogdan-estimate}
\lambda^g \ \geq\ \frac{n_k}{4(n_k-1)}\Scal^g_{\min},
\ee
again for ascending dimensions of the single factors, 
and obtained the corresponding equality case.
\end{NB}
\begin{exa}
Consider a $10$-dimensional manifold that is a product of two $5$-dimensional
manifolds with parallel characteristic torsion; this defines a geometry with reducible parallel torsion. Then the $10$-dimensional twistorial estimate
of the eigenvalue of $\D^2$ reads (Corollary \ref{cor.twistorial-est})
\bdm
\lambda \ \geq\ \frac{5}{18}\Scalg_{\min} + \frac{25}{196} \|T\|^2 - \frac{15}{49}
\max (\mu^2),
\edm
where $\max (\mu^2)$ denotes the maximal eigenvalue of $T^2$. On the other side,
the twistorial eigenvalue estimate for products (Theorem \ref{thm.twistorial-est-prod}) yields
\bdm
\lambda \ \geq\ \frac{5}{4}\Scalg_{\min} - \frac{5}{16}\max (\mu^2).
\edm
One recognizes that one gets a truly different estimate that will be of interest
for large scalar curvatures.
\end{exa}
\subsection*{More comments on eigenvalues for product manifolds}
We shall now explain how  the proof of the main  Theorem in \cite{Alexandrov07}
can easily be modified to obtain a different kind of Riemannian estimate,
which we find of interest on its own. It is plain that their 
main result (equation 
($\ref{eq.bogdan-estimate}$)) and our Theorem \ref{thm.twistorial-est-prod}
contain a built-in asymmetry: The top dimension $n_k$
plays a particular role in the estimate and the equality case predicts a
twistor spinor on the $n_k$-dimensional factor and a parallel spinor on all
other factors. We asked ourselves: Is there an estimate (in the Riemannian
case resp.~torsion case) that is symmetric in all dimensions $n_i$ and that
reaches equality if and only if there exists a twistor spinor on each factor?

Let's recall the global setting for the Riemannian case: 
Consider a compact spin manifold $M$
whose tangent bundle $\T M$ splits into parallel, pairwise orthogonal
distributions $\T_i, i=1,\ldots, k$, and where we assume that
the dimensions are ordered by $1\leq n_1\leq\ldots\leq n_k$.
Then the partial Levi-Civita derivatives $\nabla^{g,i}$, the partial
Riemannian Dirac operators $D^g_i$, partial Riemannian curvatures etc.~are
defined as  we did before, but without any torsion appearing. 
In particular, it is true that a $D^g$ eigenspinor can be chosen in
such a way that it is also a $D^g_i$ eigenspinor for all $i$, and the
eigenvalues of the squares of these operators satisfy 
$\lambda^g=\lambda^g_1+\ldots +\lambda^g_k$. Since the manifold is locally
a product, the partial eigenvalues $\lambda^g_i$ have really a geometric meaning
of their own. 
\begin{thm}\label{thm.improvement-BK}
The eigenvalues $\lambda^g$ of $(D^g)^2$ and $\lambda^g_i$ of $(D^g_i)^2$ satisfy
the inequality
\bdm
\lambda^g - \sum_{i=1}^k \frac{\lambda^g_i}{n_i} \  \geq\ \frac{\Scalg_{\min}}{4},
\edm
and equality is obtained if and only if there exists a Riemannian 
Killing spinor on each factor of $M$.
\end{thm}
\begin{proof}
For the adapted Riemannian twistor  
operator $P^0=:P$ (compare equation (\ref{eq.adapted-twistor})) and an 
arbitrary spinor field $\psi$, B.~Alexandrov
proves in \cite{Alexandrov07} the integral formula
\bdm
\|P\psi\|^2\ =\ \left[1-\frac{1}{n_k}\right] \langle (D^g)^2\psi, \psi\rangle
-\sum_{i=1}^{k-1} \left[\frac{1}{n_i}-\frac{1}{n_k}\right]
\langle (D^g_i)^2\psi, \psi\rangle -\langle \frac{\Scalg}{4} \psi,\psi\rangle.
\edm
In order to obtain a result (our equation (\ref{eq.bogdan-estimate})) in
which  the partial eigenvalues $\lambda_i$ do not appear anymore,
he observes that $1/n_i-1/n_k\geq 0$ by assumption and  neglects these terms
together with the twistor term $\|P\psi\|^2$. For proving our Theorem, the
main point is to keep these terms! Choose an eigenspinor $\psi$ 
such that $(D^g)^2\psi=\lambda^g\psi,\ (D^g_i)^2\psi=\lambda^g_i\psi$ (this is
always possible). Since $\|P\psi\|^2\geq 0$, we obtain
\bdm
0 \ \leq\ \left[1-\frac{1}{n_k}\right] \lambda^g \|\psi\|^2 -
\sum_{i=1}^{k-1} \left[\frac{1}{n_i}-\frac{1}{n_k}\right]\lambda^g_i \|\psi\|^2
-\langle \frac{\Scalg}{4} \psi,\psi\rangle.
\edm
Estimating the scalar curvature as usual by its minimum and
and dividing by the length $\|\psi\|^2>0$ yields then the result after a
short computation, which we omit. The discussion of the equality case
follows
the same line of arguments as in \cite{Alexandrov07}.
\end{proof}
\noindent
Let us discuss the value of this result. On every single factor,
Friedrich's classical estimate
\bdm
\lambda^g_i\ \geq\ \frac{n_i}{4(n_i-1)}(\Scal^g_i)_{\min}
\edm
holds, so by summation (and a quick calculation) we get
\bdm
\lambda^g - \sum_{i=1}^k \frac{\lambda^g_i}{n_i} \  \geq\ \frac{1}{4}
\sum_{i=1}^k (\Scal^g_i)_{\min}.
\edm
Since in general $\sum (\Scal^g_i)_{\min}\leq (\sum \Scal^g_i)_{\min}$, our
result is non trivial. This is particularly plain when only 
$\Scal^g_{\min}$ is strictly positive, but not all
$(\Scal^g_i)_{\min}$. From an aesthetic point of view, Theorem 
\ref{thm.improvement-BK} removes the asymmetry of equation 
($\ref{eq.bogdan-estimate}$) and Theorem \ref{thm.twistorial-est-prod},
as desired. 

Unfortunately, Theorem \ref{thm.improvement-BK} has no reasonable
analogue for Dirac operators with torsion. In oder to obtain a similar
result, one is lead to allow the parameter $s$ of the adapted twistor operator
to change in each summannd $\T_i$. The resulting inequality then links
the eigenvalues of $\D^2$ and $(D_i^{s_i})$, where now $s_i= (n_i-1)/4(n_i-3)$.
However, these operators do not have an intrinsic geometric meaning, nor
do they commute with  $\D^2$. Hence, the result is not of interest.

\appendix\section{Proof and application of the integrability condition}
\label{app.int-cond}\noindent
%
We begin with a remarkable identity that relates the curvature operator and
the Ricci operator in the spin bundle. Recall that the curvature operator
of any spin connection can be understood as a endomorphism-valued $2$-form,
\bdm
\mathcal{R}(X,Y)\psi \ =\ \nabla_X\nabla_Y \psi - \nabla_Y\nabla_X\psi -
\nabla_{[X,Y]}\psi.
\edm
One checks that it is related to the curvature operator on $2$-forms defined through
\bdm
\mathcal{R}(e_i\wedge e_j)\ :=\ \sum_{k<l} R_{ijkl}  \, e_k\wedge e_l
\edm
by the relation
\bdm
\mathcal{R}(X,Y)\psi \ =\ \frac{1}{2}\, \mathcal{R}(X\wedge Y)\cdot\psi.
\edm
Furthermore, we understand the Ricci tensor as an endomorphism on the tangent
bundle. Then the identity stated in the following theorem is crucial for deriving
integrability conditions. It generalizes a well-known
result of Friedrich \cite{Friedrich80} for the Levi-Civita connection ($T=0$).
A special case of the result--that is, applied to a \emph{parallel} spinor--may
be found in \cite{Friedrich&I1}. The result appeared for the first time in
the diploma thesis of Mario Kassuba \cite{Kassuba06}, which was written under
the supervision of the first author and Thomas Friedrich at
Humboldt University Berlin in 2006.
\begin{thm}\label{thm.Ric-id}
%
Let $\nabla^c$ be a metric spin connection with parallel torsion $T$, $\nabla^c T=0$. Then, the following identity holds for any spinor field $\psi$ and any vector field
$X$
\bdm
{\Ric}^c(X)\cdot\psi \ =\ -2\sum_{k=1}^{n}e_k
 \mathcal{R}^c(X,e_k)\psi +\frac{1}{2}X\haken dT\cdot\psi.
\edm
\end{thm}
\begin{proof}
Rewrite the first term on the right hand side (without the numerical factor) as
\bdm
\sum_{k=1}^n e_k \mathcal{R}^c(e_l,e_k)\ =\ \frac{1}{2}\sum_{k=1}^n e_k\cdot
\mathcal{R}^c(e_l\wedge e_k)\ =\ \frac{1}{2}\sum_{k=1}^n \sum_{i<j} R^c_{lkij}
e_k e_i e_j \ =:\ R_1 + R_2,
\edm
where $R_1$ denotes all terms with three different indices $k,i,j$, and $R_2$
all terms with at least one repeated index. We first discuss $R_1$:
\bdm
R_1 \ =\ \frac{1}{2} \sum_{i<j} \left[\sum_{k<i}R^c_{lkij}e_k e_i e_j+
\sum_{i<k<j}R^c_{lkij}e_k e_i e_j+ \sum_{j<k}R^c_{lkij}e_k e_i e_j \right]\ =\
\frac{1}{2} \sum_{k<i<j} \stackrel{k,i,j}{\mathfrak{S}} \, R^c_{lkij} e_k e_i e_j ,
\edm
where the symbol $\mathfrak{S}$ denotes the cyclic sum. The first Bianchi identity
for a metric connection with parallel skew torsion \cite{Friedrich&I1},
\cite{Agricola06}
\bdm
\stackrel{X,Y,Z}{\mathfrak{S}} \mathcal{R}(X,Y,Z,V)\  = \ \frac{1}{2} dT(X,Y,Z,V)
\edm
implies then $R_1= -  e_l\haken dT/4 $. We now consider $R_2$. Here, the argument
does not depend on the detailed type of the connection, only the property of being
metric is used (it implies that $R^c_{ikjl}$ is antisymmetric in the third and fourth
argument, see \cite[Section 2.8]{Agricola06}). One checks that
\bdm
R_2 \ =\ -\frac{1}{2} \sum_{r=1}^n \left[ \sum_{p=1}^{r-1} R^c_{lppr}e_e+
\sum_{q=r+1}^n R^c_{lqqr}e_r\right].
\edm
But since the Ricci tensor is exactly the contraction of the curvature,
$R_2= - \Ric^c(e_l)/2$. This ends the proof.
\end{proof}
We use this result to formulate the necessary curvature
integrability conditions for Killing spinors with torsion.
For Riemannian Killing spinors ($T=0$), this results just means
that the underlying manifold has to be Einstein.
\begin{thm}\label{thm.int-cond-KS}
%
Suppose $\nabla^c T=0$. Let $\psi$ be a Killing spinor with torsion
with Killing number $\kappa$,   set $\lambda:=\frac{1}{2(n-3)}$
for convenience, and recall that $s=\frac{n-1}{4(n-3)}$.
Then the Ricci curvature of the characteristic connection satisfies
for all vector fields $X$ the identity
\bea[*]
{\Ric}^c(X)\psi&=&-16s\kappa(X\haken
T)\psi+4(n-1)\kappa^2X\psi+(1-12\lambda^2)(X\haken\sigma_T)\psi+\\
                        & &+2(2\lambda^2+\lambda)\sum e_k(T(X,e_k)\haken T)\psi \,.
\eea[*]
\end{thm}
\begin{proof}
We will first establish a relation between the actions of $\mathcal{R}^c$ and
$\mathcal{R}^s$ on the Killing spinor $\psi$.
As an abbreviation, we set $\lambda:=\frac{1}{2(n-3)}$.
For the curvature endomorphism, we obtain by using the identity
$\nabla^s_X\psi=\nabla^c_X\psi+\lambda (X\haken T)\psi$
and the product rule for the covariant derivative of Clifford products:
\bea[*]
\mathcal{R}^s(X,Y)\psi&=&\nabla^s_X\nabla^s_Y\psi-\nabla^s_Y\nabla^s_X\psi-\nabla^s_{[X,Y]}\psi\\
            &=& \mathcal{R}^c(X,Y)\psi +\lambda\nabla^c_X((Y\haken
T)\psi)+\lambda(X\haken T)\nabla^c_Y\psi +\lambda^2 (X\haken
T)(Y\haken T)\psi\\
            & & -\lambda\nabla^c_Y((X\haken T)\psi)-\lambda(Y\haken T)\nabla^c_X\psi
-\lambda^2 (Y\haken T)(X\haken T)\psi -\lambda([X,Y]\haken T)\psi\\
            &=& \mathcal{R}^c(X,Y)\psi +\lambda(\nabla^c_X((Y\haken
T))\psi+\lambda(Y\haken T)\nabla^c_X\psi +\lambda (X\haken
T)\nabla^s_Y\psi\\
            & & -\lambda (\nabla^c_Y(X\haken T))\psi -\lambda(X\haken
T)\nabla^c_Y\psi - \lambda(Y\haken T)\nabla^s_X\psi -\lambda([X,Y]\haken T)\psi\\
            &=& \mathcal{R}^c(X,Y)\psi +\lambda(\nabla^c_X Y\haken
T)\psi+\lambda^2(X\haken T)(Y\haken T)\psi\\
            & & -\lambda(\nabla^c_Y(X\haken T))\psi-\lambda^2(Y\haken T)(X\haken
T)\psi -\lambda([X,Y]\haken T)\psi\,.
\eea[*]
From the general formula $\nabla_X(Y\haken \omega)=(\nabla_XY)\haken \omega+Y\haken
(\nabla_X \omega)$ and the assumption $\nabla^cT=0$, we conclude:
\bea[*]
\mathcal{R}^s(X,Y)\psi&=&\mathcal{R}^c(X,Y)\psi+\lambda^2(X\haken T)(Y\haken
T)\psi-\lambda^2(Y\haken T)(X\haken T)\psi\\
                      & &+\lambda(T(X,Y)\haken T)\psi\,.
\eea[*]
Together with the identities $(2)$ and $(4)$ from the compilation
of important formulas (Lemma \ref{lem.compilation}),
this allows us to compute the summand $\sum
e_k\mathcal{R}^s(X,e_k)\psi$,
\bea[*]
\sum e_k\mathcal{R}^s(X,e_k)\psi&=& \sum
e_k\mathcal{R}^c(X,e_k)\psi+3\lambda^2((X\haken T)T-T(X\haken T))\psi\\
                                & & -2\lambda^2\sum T(X,e_k)(e_k\haken T)\psi + \lambda \sum
e_k(T(X,e_k)\haken T)\psi\,.
\eea[*]
The second term can be simplified through formulas $(2)$ and $(6)$ of
Lemma \ref{lem.compilation},
\bdm
(X\haken T)T-T(X\haken T)=(XT^2+T^2X)= -\frac{1}{2}X \sigma_T-\sigma_T
X=-2X\haken\sigma_T\,.
\edm
The third term can be simplified as follows,
\bea[*]
\sum_k T(X,e_k)(e_k\haken T)&=&\sum_{k, m} T(X,e_k,e_m)e_m(e_k\haken T)\\
                             &=&\sum_m e_m \sum_k T(X,e_k, e_m) (e_k\haken T)\\
                             &=&-\sum_m e_m \sum_k T(X,e_m, e_k (e_k\haken T)\\
                             &=&-\sum_m e_m (T(X,e_m)\haken T)\,.\eea[*]
Thus, we obtain alltogether:
\bdm
\sum e_k\mathcal{R}^s(X,e_k)\psi=\sum
e_k\mathcal{R}^c(X,e_k)\psi-6\lambda^2(X\haken\sigma_T)\psi+(2\lambda^2+\lambda)\sum
e_k(T(X,e_k)\haken T)\psi\,.
\edm
Now we specialize to the case that $\psi$ is a Killing spinor with
$\nabla_X^s\psi=\kappa X\psi$. In this case,
$\mathcal{R}^s(X,Y)$ acts on $\psi$ by
\bea[*]
\kr^s(X,Y)\psi&=&\nabla^s_X\nabla^s_Y\psi-\nabla^s_Y\nabla^s_X\psi-\nabla^s_{[X,Y]}\psi\\
            &=& \kappa T^s(X,Y)\psi + \kappa^2(YX-XY)\psi\,,
\eea[*]
where $T^s=4sT$. We form again the desired sum. The relation
\bdm
\sum e_kT(X,e_k)\psi=-\sum T(X,e_k)e_k\psi=2(X\haken T)\psi
\edm
yields
\bdm
\sum e_k\mathcal{R}^s(X,e_k)\psi=8s\kappa(X\haken T)\psi+2(1-n)\kappa^2X\psi\,.
\edm
The claim now  follows from Theorem \ref{thm.Ric-id} if one observes that $dT=2\sigma_T$ holds for parallel torsion.
\end{proof}
By contracting the identity for the Ricci curvature once more, one obtains
a formula for the scalar curvature of a metric admitting Killing spinors with
torsion. However, one checks that this result coincides with the equation for
the scalar curvature stated in Lemma \ref{gen-Killing-eq}.

We now give a typical example how the previous result can be used to prove
non-existence results for Killing spinors with torsion.
\begin{cor}\label{cor.noKSonES}
A $5$-dimensional Einstein-Sasaki manifold $(M,g,\xi,\eta,\vphi)$
endowed with its characteristic
connection cannot admit Killing spinors with torsion.
\end{cor}
\begin{proof}
It is known that a $5$-dimensional Einstein-Sasaki manifold admits a local
frame such that
\bdm
\xi\ \cong\ \eta\ =\ e_5, \quad \ d\eta\ =\ 2(e_1\wedge e_2+ e_3\wedge
e_4),\quad
T^c\ =\ \eta\wedge d\eta\ =\ 2(e_1\wedge e_2+ e_3\wedge e_4)\wedge e_5.
\edm
Furthermore, in this frame, $\Scal^g=20$, $\|T\|=8$ and the eigenvalues of
$T$ are $0,\pm 4$.  Hence, Lemma \ref{gen-Killing-eq} allows us to compute
all possible values of the Killing number $\kappa$, leading to the table
\bdm
\begin{array}{|c|c|c|c|}\hline
\mu & 0^{\phantom{1^1}} & 4 & -4 \\[1mm] \hline
\kappa & \pm \frac{1}{2} & -1 \pm \frac{\sqrt{5}^{\phantom{1^1}}}{10} & 1 \pm \frac{\sqrt{5}}{10}\\[1mm]
\hline
\end{array}
\edm
In this situation, $s=1/2$ and $\lambda=1/4$. For the Ricci curvature,
observe that these are related by
\bdm
\Ric^g(X,Y)\ =\ \Ric^c(X,Y) +\frac{1}{4}\sum_{i=1}^5 g(T^c(X,e_i), T^c(Y,e_i)).
\edm
Thus, the Einstein condition $\Ric^g(e_i)=4\cdot \Id$ implies
$\Ric^c(e_i)=2\cdot \Id$ for $i=1,\ldots,4$, $\Ric^c(e_5)=0$.
Now pick your favorite spin
representation and evaluate with arbitrary $\kappa$ and for $X=e_1$
the expression (just a $(4\x 4)$-matrix)
\bdm
2\cdot \Id  -\left[-8\kappa(e_1\haken
T)+16 \kappa^2 e_1 +\frac{1}{4}(e_1\haken\sigma_T)
+\frac{3}{4} \sum_{k=1}^5 e_k(T(e_1,e_k)\haken T)\right]
\edm
and check that for all possible $\kappa$ values above, the determinant is
nonzero.
Hence, this endomorphism on the spin bundle has no kernel, that is, there
cannot exist a spinor field $\phi$ (not even in a point) satisfying the
integrability condition from Theorem \ref{thm.int-cond-KS} for $X=e_1$.
\end{proof}
%
\section{Curvature properties for families of connections}\noindent
%
The curvature of a metric connection $\nabla$ 
with parallel torsion $T\in\Lambda^3(TM^n)$ is known to have some
special properties. In this section, we show how some of these
properties can be transferred to a $1$-parameter family of connections
in which only one connection has this property, matching of course
exactly the situation encountered in this paper. 

\begin{thm}\label{thm.curv-splitting}
Assume that $(M,g)$ carries a $1$-parameter family of metric connections
$\nabla^s$ with skew torsion $T\in\Lambda^3(M)$
\bdm
\nabla^s_X Y \ =\ \nabla^g_X Y + 2s\, T(X,Y,-),
\edm
and that $\nabla^{c}T=0$, where $\nabla^c$ is the connection
corresponding to $s=1/4$. Then, for all $s\in\R$,
the covariant derivative of the torsion is given  by
\bdm
\nabla^s_X T(U,V,W)\ =\ \left[2s-\frac{1}{2}\right]\, \sigma_T(U,V,W,X),
\edm
and  the first Bianchi identity reduces to
\bdm
\stackrel{X,Y,Z}{\mathfrak{S}} \mathcal{R}^s(X,Y,Z,V)\ =\
s \left[6-8s\right]\, \sigma_T(X,Y,Z,V).
\edm
This implies, in particular, that the curvature is symmetric
under intertwining of blocks,
\bdm
\mathcal{R}^s(X,Y,U,V)\ =\ \mathcal{R}^s(U,V,X,Y).
\edm
Furthermore, the  Ricci tensor $\Ric^s$ is symmetric.
\end{thm}
\begin{proof}
The connections are related by
\bdm
\nabla^s_X Y \ =\ \nabla^{c}_X Y +  \left[2s-\frac{1}{2}\right] T(X,Y,-),
\edm
thus the covariant derivatives satisfy
\bea[*]
\nabla^s_X T(U,V,W) &=& \nabla^c_X T(U,V,W)\\
&& -  \left[2s-\frac{1}{2}\right]\left[ T(T(X,U),V,W) + T(U, T(X,V), W)+ T(U,V, T(X,W)) \right]\\
&=& 0 + \left[\frac{1}{2}-2s\right] \left[- T(W,V,T(X,U)) -  T(U, W, T(X,V))+ T(U,V, T(X,W))
\right]\\
&=&  \left[\frac{1}{2}-2s\right] \left[- g(T(W,V),T(X,U)) -  g(T(U, W), T(X,V))+ g(T(U,V), T(X,W))
\right]\\
&=& \left[\frac{1}{2}-2s\right] \left[- g(T(V,W),T(U,X)) -  g(T(W,U), T(V,X))- g(T(U,V), T(W,X))
\right]\\
&=& \left[\frac{1}{2}-2s\right] \left[- \sigma_T(U,V,W,X) \right]
\ =\ \left[2s-\frac{1}{2}\right] \,\sigma_T(U,V,W,X)
\eea[*]
by the definition of $\sigma_T$ \cite[Dfn A.1]{Agricola06}.
Consider the first Bianchi identity  \cite[Thm 2.6]{Agricola06}
\bdm
\stackrel{X,Y,Z}{\mathfrak{S}} \mathcal{R}^s(X,Y,Z,V)\ =\ dT^s(X,Y,Z,V)+
\nabla^s_V T^s(X,Y,Z) - \sigma_{T^s}(X,Y,Z,V),
\edm
where $T^s=4sT$ and $\sigma_{T^s}= 16 s^2 \sigma_T$ are the corresponding
quantities of the connection $\nabla^s$. Since $\nabla^cT=0$, its
torsion $T$  satisfies $d T = 2 \sigma_{T^c}= 2\sigma_T$.
A routine calculation yields then the claimed formula.
The property of the curvature tensor  follows by the same
symmetrization argument as in \cite[Remark 2.3]{Agricola06}.

For the symmetry of the Ricci tensor,
we argue as follows: for alle parameters $s$, the
$\nabla^s$-divergences $\delta^s$ of $T$ coincide, $\delta^s T=\delta^g T$
(see \cite[Prop. A.2]{Agricola06}). But  $\nabla^c T=0$ implies
$\delta^c T=0$, so $\delta^s T=0$, and  this
is precisely the antisymmetric part of the Ricci tensor, so
symmetry follows at once (\cite{Friedrich&I1}, \cite[Thm A.1]{Agricola06}).
\end{proof}
%
\section{Compilation of important formulas}\label{formulas}\noindent
%
We compile some remarkable identities that are used throughout
this article. All of them are routine exercises, so we abstain
from giving proofs or detailed references for all.
The less obvious formulas ($4$-$6$) can be found in
\cite{Bismut} and \cite{Friedrich&I1} (see also \cite{Agricola06}), though 
even earlier publications could be  possible.
Recall the definition of the important $4$-form $\sigma_T$
derived from any $3$-form $T$:
\bdm
\sigma_T\ :=\ \frac{1}{2}\sum_i (e_i\haken T)\wedge (e_i\haken T).
\edm
\begin{lem}\label{lem.compilation}
For a $3$-form $T$, a $k$-form $\omega$, a vector field $X$, an orthonormal frame
$e_1,\ldots, e_n$  and spinor fields  $\psi,\vphi$, the following identities
hold:
\begin{enumerate}
\item $X\cdot T=X\wedge T-X\haken T,\quad T\cdot X = -X\wedge T-X\haken T$
\item $X\cdot T+ T\cdot X = -2 X\haken T$, more generally,
$X\cdot \omega -(-1)^k \omega\cdot X = -2X\haken \omega$.
\item $\langle T\cdot\psi,\vphi\rangle = \langle \psi, T\cdot \vphi\rangle $
\item $\displaystyle\sum_{i=1}^n (e_i\haken T)e_i =\sum_{i=1}^n e_i (e_i\haken T) = 3 T$
\item $\displaystyle\sum_{i=1}^n (e_i\haken T)\cdot (e_i\haken T)=
\sum_{i=1}^n (e_i\haken T)\wedge (e_i\haken T) -3\|T\|^2=2\sigma_T-3\|T\|^2 $
\item $\displaystyle T^2 = -\sum_{i=1}^n (e_i\haken T)\wedge (e_i\haken T)
+\|T\|^2 = -2\sigma_T+\|T\|^2$
\item $\displaystyle\sum_{i=1}^n e_i \cdot (e_i\wedge T) = (3-n)T$
\item $\displaystyle \sum_{j=1}^n T(X,e_j)\cdot e_j = -2 X\haken T $
\end{enumerate}
\end{lem}
%



\begin{thebibliography}{1111}
%
\bibitem[Ag03]{Agricola03}
I. Agricola, \emph{Connections on naturally reductive spaces, their
Dirac operator and homogeneous models in string theory}, Comm.
Math. Phys. 232 (2003), 535-563.
%
\bibitem[Ag06]{Agricola06}
I.\,Agricola, \emph{The Srn\'{\i} lectures on non-integrable geometries with
torsion}, Arch. Math. (Brno) 42 (2006), 5--84. With an appendix by
Mario Kassuba.
%
\bibitem[AF04a]{Agricola&F04a}
I.\,Agricola and Th.\,Friedrich, \emph{On the holonomy of connections with
skew-symmetric torsion}, Math. Ann. 328 (2004), 711-748.
%
\bibitem[AF04b]{Agricola&F04b}
\bysame, \emph{The Casimir operator of a metric
connection with skew-symmetric torsion}, Jour. Geom. Phys. 50 (2004), 188-204.
%
%
%
\bibitem[AFK08]{Agricola&F&K08}
I. Agricola, Th. Friedrich, and M. Kassuba,
\emph{Eigenvalue estimates for Dirac operators with parallel characteristic
torsion}, Diff. Geom. Appl. 26 (2008), 613-624.
%
\bibitem[Al03]{Alexandrov03}
B. Alexandrov, \emph{$Sp(n)U(1)$-connections with parallel totally
skew-symmetric torsion},  Journ. Geom. Phys. 57 (2006), 323-337.
%
\bibitem[Al07]{Alexandrov07}
\bysame, \emph{The first eigenvalue of the Dirac operator on
locally reducible Riemannian manifolds},
J. Geom. Phys. 57 (2007), 467--472.
%
\bibitem[AlFS04]{Alexandrov&F&S04}
B. Alexandrov, Th. Friedrich, N. Schoemann, \emph{Almost Hermitian
$6$-manifolds revisited}, J. Geom. Phys. 53 (2005), 1-30.
%
\bibitem[AlGI98]{Alexandrov&G&I98}
B. Alexandrov, G. Grantcharov, and S. Ivanov, \emph{An estimate for the
first eigenvalue of the Dirac operator on compact Riemannian spin
manifold admitting parallel one-form}, J. Geom. Phys. 28 (1998), 263--270.
%
\bibitem[AI00]{Alexandrov&I00}
B. Alexandrov, S. Ivanov, \emph{Dirac operators on Hermitian spin surfaces},
Ann.~Global Anal.~Geom. 18 (2000), 529-539.
%
\bibitem[BFGK91]{BFGK}
H. Baum, Th. Friedrich, R. Grunewald, I. Kath, \emph{Twistors and Killing
spinors on Riemannian manifolds}, Teubner-Texte zur Mathematik, Band 124,
Teubner-Verlag Stuttgart / Leipzig, 1991.
%
\bibitem[Be12]{BB12} J. Becker-Bender, \textit{Dirac-Operatoren und 
Killing-Spinoren mit Torsion}, Ph.D. Thesis, University of Marburg (2012).
%
\bibitem[BM01]{Belgun&M01}
F. Belgun, A. Moroianu,
\emph{Nearly K\"ahler 6-manifolds with reduced holonomy},
Ann. Global Anal. Geom. 19 (2001), 307--319.
%
\bibitem[Bes87]{Besse87}
A.~Besse, \emph{Einstein manifolds}, Ergebnisse der Mathematik und ihrer
Grenzgebiete Bd.~10, Springer-Verlag Berlin-Heidelberg 1987.
%
\bibitem[Bi89]{Bismut}
J. M. Bismut, \emph{A local index theorem for non-K\"ahlerian manifolds},
Math. Ann. 284 (1989), 681-699.
%
\bibitem[CM11]{Cleyton&Moroianu}
R. Cleyton, A. Moroianu, \emph{Metric connections with parallel torsion},
to appear.
%
\bibitem[DI01]{Dalakov&I01}
P.Dalakov, S. Ivanov, \emph{Harmonic spinors of Dirac operator of connection 
with torsion in dimension $4$}, 
Class. Quant. Grav. 18 (2001), 253-265.

\bibitem[Fr80]{Friedrich80}
Th. Friedrich, \emph{Der erste {E}igenwert des {D}irac-{O}perators einer
  kompakten, {R}iemannschen {M}annigfaltigkeit nichtnegativer
  {S}kalarkr{\"u}mmung}, Math. Nachr. 97 (1980), 117-146.
%
\bibitem[Fr89]{Friedrich89}
\bysame, \emph{On the conformal relations between twistor and Killing
  spinors}, Suppl. Rend. Circ. Mat. Palermo (1989), 59-75.
%
\bibitem[Fr02]{Dirac-Buch}
\bysame, \emph{Dirac operators in Riemannian geometry},
Graduate Studies in Mathematics 25, AMS, Providence, Rhode Island, 2000.
%
%
\bibitem[Fri03b]{Fri2}
\bysame, \emph{On types of non-integrable geometries}, 
Suppl. Rend. Circ. Mat. di Palermo Ser. II, 71 (2003), 99-113.
%
\bibitem[Fr07a]{Fr07a}
\bysame, \emph{$G_2$-manifolds with parallel characteristic torsion},
Diff. Geom. Appl. 25 (2007), 632--648.
%
\bibitem[FG85]{Friedrich&G85}
Th. Friedrich and R. Grunewald, \emph{On the first eigenvalue of the
Dirac operator on $6$-dimensional manifolds}, Ann Glob. Anal. Geom. 3
(1985), 265-273.
%
\bibitem[FrI02]{Friedrich&I1}
Th.\,Friedrich and S.\,Ivanov, \emph{Parallel spinors and connections with
  skew-symmetric torsion in string theory}, Asian Journ. Math. 6 (2002),
303-336.
%
\bibitem[FI03a]{Friedrich&I2}
\bysame, \emph{Almost contact manifolds, connections with torsion and
parallel spinors}, J. Reine Angew. Math. 559 (2003), 217-236.
%
\bibitem[FKMS97]{FKMS}
Th. Friedrich, I. Kath, A. Moroianu and U. Semmelmann, 
\emph{On nearly parallel $\mathrm{G}_2$-structures}, Journ. Geom. Phys.
23 (1997), 256-286.
%
\bibitem[Gau97]{Gauduchon97}
P. Gauduchon, \emph{Hermitian connections and Dirac operators},
Boll. Un. Mat. Ial. ser. VII  2 (1997), 257-289.
%
\bibitem[GO98]{Gauduchon&O98}
P. Gauduchon, L. Ornea, \emph{Locally conformally K\"ahler metrics on
Hopf surfaces}, Ann. Inst. Fourier 48 (1998), 1107-1127.
%
%
\bibitem[Gra71]{Gray71} 
A. Gray, \emph{Weak holonomy groups}, Math. Z. 123 (1971), 
290-300.
%
\bibitem[Ha90]{Habermann90}
K. Habermann, \emph{The twistor equation on Riemannian spin manifolds},
J. Geom. Phys.(1990), 469-488.
%
\bibitem[Hi98]{Hijazi98}
O. Hijazi, \emph{Twistor operators and eigenvalues of the Dirac operator},
Gentili, G. (ed.) et al., Quaternionic structures in mathematics 
and physics. Proceedings of the meeting, Trieste, Italy, September 5-9, 1994. 
Trieste: International School for Advanced Studies (SISSA),  151-174 (1998).


\bibitem[HL88]{Hijazi&L88}
O. Hijazi, A. Lichnerowicz, \emph{Spineurs harmoniques,
spineurs-twisteurs et g\'eom\'etrie conforme}, C. R. Acad. Sci. Paris 307,
S\'erie I (1988), 833-838.
%
\bibitem[Hit74]{Hitchin74} 
N. Hitchin, \emph{Harmonic spinors}, Adv. in Math. 14 (1974), 1-55. 
%
\bibitem[Hit00]{Hitchin00} 
\bysame, \emph{The geometry of three-forms in six and seven dimensions}, 
Journ. Diff. Geom. 55 (2000), 547-576.
%
\bibitem[HKWY10]{Houri&KWY10}
T. Houri, D. Kubiznak, C. Warnick, Y. Yasui,
\emph{Symmetries of the Dirac operator with
skew-symmetric torsion},  Class. Quantum Grav. 27 (2010), 185019.
%
\bibitem[Jen75]{Jensen75}
G.~Jensen, \emph{Imbeddings of {S}tiefel manifolds into {G}rassmannians}, Duke
  Math. J. 42 (1975), 397--407.
%
\bibitem[Ka06]{Kassuba06}
M. Kassuba, \emph{Der erste Eigenwert des Operators $D^{1/3}$ einer kompakten
Sasaki-Mannigfaltigkeit}, diploma thesis, Humboldt Univ. Berlin, 2006.
%
\bibitem[Ka10]{Kassuba10}
M. Kassuba, \emph{Eigenvalue estimates for Dirac operators in geometries with
  torsion}, Ann. Glob. Anal. Geom. 37 (2010), 33-71.
%
\bibitem[Ka00]{Kath00}
I. Kath, \emph{Pseudo-Riemannian $T$-duals of compact Riemannian
homogeneous spaces}, Transform. Groups 5 (2000), 157-179.
%
\bibitem[Ki04]{ECKim04}
E.\,C. Kim, \emph{Lower bounds of the Dirac eigenvalues on Riemannian product
manifolds}, math.DG/0402427
%
\bibitem[KN63]{Kobayashi&N1}
S. Kobayashi and K. Nomizu, \emph{Foundations of differential geometry {I}},
  Wiley Classics Library, Wiley Inc., Princeton, 1963, 1991.
%
\bibitem[KN69]{Kobayashi&N2}
\bysame, \emph{Foundations of differential geometry {II}},
  Wiley Classics Library, Wiley Inc., Princeton, 1969, 1996.
%
\bibitem[Ko99]{Kostant99}
B. Kostant, \emph{A cubic {D}irac operator and the emergence of {E}uler number
  multiplets of representations for equal rank subgroups}, Duke Math. J.
 100 (1999), 447-501.
%
\bibitem[Li87]{Lichnerowicz87}
A. Lichnerowicz, \emph{Spin manifolds, Killing spinors and universality of the
Hijazi inequality}, Lett. Math. Phys. 13 (1987), 331-344.
%
\bibitem[Li88]{Lichnerowicz88}
\bysame, \emph{Killing spinors, twistor-spinors and Hijazi inequality},
J. Geom. Phys. 5 (1988), 2-18.
%
\bibitem[Sal89]{Salamon89}
S. Salamon, \emph{Riemannian geometry and holonomy groups}, 
Pitman Research Notes in Mathematical Series, 201. Jon Wiley \& Sons, 1989.
%
\bibitem[Se98]{Semmelmann98}
U. Semmelmann, \emph{A short proof of eigenvalue estimates for the Dirac 
operator on Riemannian and K\"ahler manifolds},
Differential Geom. Appl., proceedings, Brno (1998), 137-140.

\bibitem[Sch07]{Schoemann}
Nils Schoemann,
\emph{Almost Hermitian structures with parallel torsion},
J. Geom. Phys. 57 (2007), 2187--2212.
%
\bibitem[Sw00]{Swann00}
A. F. Swann, \emph{Weakening holonomy}, ESI preprint No. 816
(2000); in S. Marchiafava et.~al. (eds.), Proc. of the second meeting on 
quaternionic structures in Mathematics and Physics, Roma 6-10 September 1999,
World Scientific, Singapore 2001, 405-415.
%
\bibitem[Va79]{Vaisman79}
I. Vaisman, \emph{Locally conformal  K\"ahler manifolds with parallel Lee
form}, Rend. Math. Roma 12 (1979), 263-284.
%
\end{thebibliography}
\end{document}